\documentclass[a4paper,reqno]{amsart}
\usepackage{geometry}
\usepackage{mathrsfs}
\usepackage{syntonly}
\usepackage{amsmath}
\usepackage{amsthm}
\usepackage{amsfonts}
\usepackage{amssymb}
\usepackage{latexsym}
\usepackage{amscd,amssymb,amsopn,amsmath,amsthm,graphics,amsfonts,mathrsfs,accents,enumerate,verbatim,calc}
\usepackage[dvips]{graphicx}
\usepackage[colorlinks=true,linkcolor=blue,citecolor=blue]{hyperref}
\input xy
\xyoption{all}

\usepackage{setspace}

\usepackage{color}

\numberwithin{equation}{section}

\theoremstyle{plain}

\newtheorem{thm}{Theorem}[section]
\newtheorem{cor}[thm]{Corollary}
\newtheorem{lem}[thm]{Lemma}
\newtheorem{prop}[thm]{Proposition}

\newtheorem{defn}[thm]{Definition}
\newtheorem{exm}[thm]{Example}

\newtheorem{rem}[thm]{Remark}

\newdir{ >}{{}*!/-5pt/\dir{>}}

\newcommand{\Mod}{\operatorname{Mod}\nolimits}

\newcommand{\Fac}{\operatorname{Fac}\nolimits}

\newcommand{\Hom}{\operatorname{Hom}\nolimits}

\newcommand{\Ext}{\operatorname{Ext}\nolimits}

\renewcommand{\Mod}{\mathsf{Mod}\hspace{.01in}}
\renewcommand{\mod}{\mathsf{mod}\hspace{.01in}}

\newcommand{\M}{\mathcal M}

\newcommand{\B}{\mathcal B}

\newcommand{\U}{\mathcal U}
\newcommand{\V}{\mathcal V}
\newcommand{\A}{\mathcal A}

\newcommand{\T}{\mathcal T}
\def\RR{\mathcal{R}\ast\mathcal{R}[1]}

\newcommand{\D}{\mathcal D}

\newcommand{\N}{\mathcal N}
\newcommand{\R}{\mathcal R}
\newcommand{\X}{\mathscr X}
\newcommand{\Y}{\mathscr Y}

\newcommand{\C}{\mathcal C}

\newcommand{\E}{\mathcal E}

\newcommand{\svecv}[2]{\left(\begin{smallmatrix}
      #1 \\
      #2
    \end{smallmatrix}\right)}

\newcommand{\svech}[2]{\left(\begin{smallmatrix}
      #1 & #2
\end{smallmatrix}\right)}

\renewcommand{\emph}{\textit}
\renewcommand{\phi}{\varphi}

\newcommand{\add}{\mathsf{add}\hspace{.01in}}

\begin{document}

\title{Relative cluster tilting theory and $\tau$-tilting theory}\footnote{\hspace{1em}Yu Liu is supported by the National Natural Science Foundation of China (Grant No. 12171397). Panyue Zhou is supported by the National Natural Science Foundation of China (Grant No. 12371034) and by the Hunan Provincial Natural Science Foundation of China (Grant No. 2023JJ30008). }
\author{Yu Liu, Jixing Pan and Panyue Zhou}
\address{ School of Mathematics and Statistics, Shaanxi Normal University, 710062 Xi'an, Shannxi, P. R. China}
\email{recursive08@hotmail.com}

\address{Department of Mathematical Sciences, Tsinghua University, 100084 Beijing, P. R. China}
\email{pjx19@tsinghua.org.cn}

\address{School of Mathematics and Statistics, Changsha University of Science and Technology, 410114 Changsha, Hunan, P. R. China}
\email{panyuezhou@163.com}

%\thanks{$^\ast$Corresponding author.}

%\thanks{The authors would like to thank Professor Bin Zhu for helpful discussions.}
\begin{abstract}
Let $\C$ be a Krull-Schmidt triangulated category with shift functor $[1]$ and $\R$ be a rigid subcategory of $\C$. We are concerned with the mutation of two-term weak $\R[1]$-cluster tilting subcategories. We show that any almost complete two-term weak $\R[1]$-cluster tilting subcategory has exactly two completions.
Then we apply the results on relative cluster tilting subcategories to the domain of $\tau$-tilting theory in functor categories and abelian categories.

%the functorial version of $\tau$-tilting theory.
\end{abstract}
\maketitle

\tableofcontents

\section{Introduction}

Tilting theory emerges as a versatile method for establishing equivalences between categories. Initially developed in the framework of module categories over finite dimensional algebras, tilting theory has now evolved into an indispensable tool across various mathematical domains. Applications have been found in diverse fields such as finite and algebraic group theory, commutative and non-commutative algebraic geometry, as well as algebraic topology.

Let $\Lambda$ be a finite dimensional hereditary algebra over an algebraically closed field.
All modules  are regarded as finite dimensional left $\Lambda$-modules.
We assume that $\Lambda$ has $n$ isomorphism classes of simple modules.
A partial tilting module $M$ is called an almost complete tilting module, if $M$
has $n-1$ non-isomorphic indecomposable direct summands.
An indecomposable module $X$ is called a complement to $M$,
if $M\oplus X$ is a tilting module and $X$ is not isomorphic to any direct summand of $M$.
It has been  shown in \cite{HR} that there always exists a complement to an almost complete tilting module $M$, and further research \cite{RS,U,HU} showed that such a partial tilting module allows (up to isomorphism) at most two complements and there are two such complements if and only if the partial tilting module is sincere.

Let $A$ be a finite dimensional algebra over a  field. Adachi, Iyama and Reiten \cite{AIR}
introduced \emph{$\tau$-tilting theory}, which is a
generalization of classical tilting theory.
The impetus to explore $\tau$-tilting theory arises from various sources, with a key focus on the mutation of tilting modules. It's worth noting that one limitation of tilting module mutation is its occasional infeasibility. This limitation serves as the driving force behind the exploration of $\tau$-tilting theory.
Support $\tau$-tilting (resp. $\tau$-rigid) modules extend the scope of tilting (resp. partial tilting) modules by incorporating the Auslander-Reiten translation. In essence, support $\tau$-tilting modules offer a natural extension of the class of tilting module with respect to mutation. Notably, as demonstrated in \cite{AIR}, a basic almost complete support $\tau$-tilting module emerges as a direct summand of exactly two basic support $\tau$-tilting modules.  This means that mutation
of support $\tau$-tilting modules is always possible.

Cluster categories were introduced by Buan, Marsh, Reineke, Reiten, and Todorov \cite{BMRRT} in order to categorify Fomin-Zelevinsky's cluster algebras \cite{FZ1,FZ2}. Consider $\Lambda$, a finite dimensional hereditary algebra over an algebraically closed field with $n$ non-isomorphic simple modules. The cluster category of $\Lambda$, denoted by $\C_\Lambda=\mathrm{D}^b({\rm mod}\Lambda)/\tau^{-1}[1]$, arises as the Verdier quotient category.
In a triangulated category, rigid objects manifest in three distinct forms: cluster tilting, maximal rigid, and complete rigid. Generally speaking, these forms are not equivalent \cite{BIKR,KZ}, yet in the cluster category $\C_\Lambda$, they are demonstrated to be equivalence \cite{BMRRT}. In contrast to classical tilting modules, cluster tilting objects in cluster categories exhibit favorable characteristics \cite{BMRRT}.
For instance, any almost complete cluster tilting object in a cluster category can be completed to a cluster tilting object in exactly two ways.
Adachi, Iyama and Reiten \cite{AIR} established a
bijection between cluster tilting objects in a $2$-Calabi-Yau triangulated category and
support $\tau$-tilting modules over a cluster-tilted algebra (also see \cite{CZZ,YZ,FGL} for various versions of this bijection).

Mutation of $n$-cluster tilting subcategories was introduced by Iyama and Yoshino \cite{IY},
they also studied almost complete $n$-cluster
tilting subcategories. They proved that any almost complete cluster tilting subcategory of a $2$-Calabi-Yau triangulated category is contained in exactly two cluster tilting subcategories.
We know that cluster tilting subcategories are maximal rigid subcategories, but the converse is
not true in general. Zhou and Zhu \cite{ZZ} proved that
any almost complete maximal rigid subcategory of a $2$-Calabi-Yau triangulated category is contained in exactly two maximal rigid subcategories.
In \cite{IJY}, Iyama, J{\o}rgensen, and Yang introduced a functorial version of $\tau$-tilting theory. They extended the concept of support $\tau$-tilting modules from the module categories of finite dimensional algebras to essentially small additive categories. Their results demonstrated that for a triangulated category $\C$ with a silting subcategory $\mathcal{S}$, there exists a one-to-one correspondence between the set of two-term silting subcategories in $\mathcal{S}*\mathcal{S}[1]$ and the set of support $\tau$-tilting subcategories of $\mod\mathcal S$.

For a Krull-Schmidt, Hom-finite triangulated category with a Serre functor
and a cluster tilting object $T$, let $\Lambda$ be the endomorphism algebra of $T$. Yang and Zhu \cite{YZ} proved that there exists a bijection between the set of isomorphism classes of basic $T[1]$-cluster tilting objects in this triangulated category and the set of isomorphism classes of basic support $\tau$-tilting $\Lambda$-modules.
Yang, Zhou, and Zhu \cite{YZZ} introduced the notion of (weak) $\mathcal{T}[1]$-cluster tilting subcategories for a triangulated category $\C$ with a cluster tilting subcategory $\mathcal{T}$. They established a correspondence between the set of weak $\mathcal{T}[1]$-cluster tilting subcategories of $\C$ and the set of support $\tau$-tilting subcategories of $\mod\mathcal T$, which extends the above correspondence by Yang and Zhu in \cite{YZ}. This correspondence resembles the one presented in \cite{IJY}.
Consider a triangulated category $\C$ with shift functor $[1]$ and a rigid subcategory $\mathcal{R}$. Zhou and Zhu \cite{ZhZ} introduced the concepts of two-term $\mathcal{R}[1]$-rigid subcategories, two-term (weak) $\mathcal{R}[1]$-cluster tilting subcategories, and two-term maximal $\mathcal{R}[1]$-rigid subcategories. They proved that there is a bijection between the set of two-term $\mathcal{R}[1]$-rigid subcategories of $\C$ and the set of $\tau$-rigid subcategories of $\mod\mathcal R$, which establishes a one-to-one correspondence between the set of two-term weak $\mathcal{R}[1]$-cluster tilting subcategories of $\C$ and the set of support $\tau$-tilting subcategories of $\mod\mathcal R$. This  result extends the main results in \cite{YZZ} when $\mathcal{R}$ is a cluster tilting subcategory. Furthermore, when $\mathcal{R}$ is a silting subcategory, they showed that the two-term weak $\mathcal{R}[1]$-cluster tilting subcategories precisely coincide with the two-term silting subcategories presented in \cite{IJY}, thus establishing a bijection analogous to the bijection in \cite{IJY}.
\vspace{2mm}

Let $\C$ be a Krull-Schmidt triangulated category with shift functor $[1]$ and $\R$ be a rigid subcategory of $\C$. We first show the following result, which is an analogue of co-Bongartz completion.

\begin{prop}\label{prop11}{\rm (see Proposition \ref{facmax} for details)}
Let $\X$ be a two-term $\R[1]$-rigid subcategory such that every object in $\R$ admits a left $\X$-approximation. Then it is contained in a two-term weak $\R[1]$-cluster tilting subcategory $\M_{\X}$.
\end{prop}

By \cite[Theorem 3.1]{ZhZ}, for any two-term $\R[1]$-rigid subcategory $\X$ such that every object in $\R[1]$ admits a right $\X$-approximation, we can construct a two-term weak $\R[1]$-cluster tilting subcategory $\mathcal N_{\X}$ which contains $\X$.
Our first main result shows that there exists a mutation for
two-term weak $\R[1]$-cluster tilting subcategories.
Let $\X$ be a subcategory of $\C$. For convenience, if each object in $\R$ admits a left $\X$-approximation and each object in $\R[1]$ admits a right $\X$-approximation, we say that $\X$ is $\R[1]$-functorially finite.

\begin{thm}{\rm (see Definition \ref{defmu} and Theorem \ref{m-pair} for details)}
Let $\X$ be an $\R[1]$-functorially finite two-term $\R[1]$-rigid subcategory. Then $(\M_\X,\mathcal N_\X)$ is an $\X$-mutation pair.
\end{thm}

Next, we introduce the definition of the \emph{almost complete} relative cluster tilting subcategories.

\begin{defn}\label{complete}
We call a two-term $\R[1]$-rigid subcategory $\X$ an almost complete two-term weak $\R[1]$-cluster tilting subcategory
if the following conditions are satisfied:
\begin{itemize}
\vspace{1mm}
\item[\rm (a)] $\X$ is $\R[1]$-functorially finite;
\vspace{1mm}

\item[\rm (b)] There exists an indecomposable object $W\notin \X$ such that $\add(\X\cup\{W\})$ is two-term weak $\R[1]$-cluster tilting.
\end{itemize}

Any two-term weak $\R[1]$-cluster tilting subcategory which has the form in {\rm (b)} is called a completion of $\X$.
\end{defn}

Our second main result concerns the existence of completions for almost complete two-term weak $\R[1]$-cluster tilting subcategories.
%{\red It is particularly emphasized that our proof is given directly in the triangulated category, without reliance on the abelian category or module category. }

\begin{thm}\label{main54}{\rm (see Theorem \ref{main1} for details)}
Any almost complete two-term weak $\R[1]$-cluster tilting subcategory $\X$ has exactly two completions  $\M_{\X}$ and $\N_{\X}$.
Moreover, if $\U$ is a two-term weak $\R[1]$-cluster tilting subcategory which contains $\X$, then $\U=\M_\X$ or $\U=\mathcal N_{\X}$.
\end{thm}

We derive the following direct consequence of Theorem \ref{main54}, where the first part is established in \cite{IY}, and the second part in \cite{ZZ}.

\begin{cor}
{\rm (a)} In a $2$-Calabi-Yau triangulated category with cluster tilting subcategories, any almost complete
cluster tilting subcategories have exactly two completions.

{\rm (b)} In a $2$-Calabi-Yau triangulated category with
maximal rigid subcategories, any almost complete
maximal rigid subcategories  {\rm (}in the sense of \cite{ZZ}{\rm)} have exactly two completions.
\end{cor}

%Our fourth main result prove the analogue, for $\tau$-rigid subcategories,
%of the completions of $\tau$-rigid modules.
%
%
%\begin{thm}\label{main16}{\rm (see Theorem \ref{thm2} for details)}
%Let $\X$ be an $\R[1]$-functorially finite two-term $\R[1]$-rigid subcategory such that $\X\nsubseteq \R[1]$.
%The following statements hold.
%\begin{itemize}
%\vspace{1mm}
%\item[\rm (1)] $\overline \M_\X=\{M\in \Fac \overline \X ~|~ {\rm Ext}^1_{\overline \A}(M,\Fac \overline \X)=0\}=:\mathbf{P}(\Fac \overline \X)$;
%\vspace{2mm}
%
%\item[\rm (2)] If $\X\cap \R[1]=0$, then $\overline {\mathcal N}_\X=\{N\in {^{\bot}(\tau \overline \X)} ~|~ {\rm Ext}^1_{\overline \A}(N, {^{\bot}}(\tau \overline \X))=0\}=:\mathbf{P}({^{\bot}}(\tau \overline \X))$.
%
%
%    We call $\mathbf{P}({^{\bot}}(\tau \overline \X))$ the Bongartz completion of $\overline{\X}$.
%\end{itemize}
%\end{thm}

Denote $\Mod\R$ by the abelian category of contravariantly additive functor from $\R$ to the category of abelian group, and $\mod\R$ for the full subcategory of
finitely presentation functor (see \cite{Au} for details). There exists a functor
\begin{eqnarray*}
\mathbb{H}\colon\C & \longrightarrow & \Mod\R \\
 M& \longmapsto & \Hom_{\C}(-,M)\mid_{\R}
\end{eqnarray*}
which is called the restricted Yoneda functor.
By \cite[Proposition 6.2]{IY}, the functor $\mathbb{H}$
    induces an equivalence
$$(\RR)/\R[1]\xrightarrow{~\simeq~}  \mod\R.$$
This result serves as a bridge to connect (relative) cluster tilting theory and $\tau$-tilting theory.

Bongartz completion is not always possible in classic tilting theory.
However, this is always possible in $\tau$-tilting theory.
Our third main result shows that every $\tau$-rigid pair
is contained in exactly two support $\tau$-tilting pairs in functor categories.

\begin{thm}\label{main17}{\rm (see Definition \ref{taupair} and Theorem \ref{main} for details)}
Assume that $(\R*\R[1])/\R[1]$ is abelian. Let $\X\vspace{1.5mm}$ be an $\R[1]$-functorially finite two-term $\R[1]$-rigid subcategory of $\C$. Denote $\X[-1]\cap\R$ by $\E$. If $\vspace{1.5mm}(\overline{\X}, \overline\E)$ is not a support $\tau$-tilting pair in $(\R*\R[1])/\R[1]$, then it is contained in two support $\tau$-tilting
pairs $(\overline{\M}_{\X}, \overline {\mathcal F})$ and $(\overline{\N}_{\X}, \overline\E)$ such that
$$\Fac\overline\M_\X=\Fac \overline \X,~ \Fac\overline\N_\X={^\perp(\tau\overline{\X})}\cap{\overline\E^{\perp}}.$$
Moreover, if $\X$ is almost complete, then $(\overline \M_\X, \overline{\R(\X)})$ and $(\overline \N_\X, \overline\E)$ are the only support $\tau$-tilting pairs that contain $(\overline{\X}, \overline\E)$.
\end{thm}

Buan and Zhou \cite{BZ} introduced the notion of left weak cotorsion torsion triples in module categories and established a bijection between support $\tau$-tilting modules and these triples.
Let $\widehat{\A}$ be an abelian
category with enough projectives. Asadollahi, Sadeghi and Treffinger \cite{AST} introduced the notion of $\tau$-cotorsion torsion
triples and established a bijection between the collection of $\tau$-cotorsion torsion triples in $\widehat{\A}$
and the collection of support $\tau$-tilting subcategories of $\widehat{\A}$.
We apply our main results to $\tau$-tilting theory in abelian categories, establishing bijections between $\tau$-cotorsion torsion pairs, left weak cotorsion torsion pairs and support $\tau$-tilting subcategories. Our fourth main result is the following.

\begin{thm}\label{main18}{\rm (see Definition \ref{lcp} and Theorem \ref{main722} for details)}
Let $R$ be a commutative noetherian ring which is complete and local. Let $\widehat{\A}$ be an {\rm Ext}-finite abelian category over $R$ with enough projectives. There are bijections between $\tau$-cotorsion torsion pairs, left weak cotorsion torsion pairs and contravariantly finite support $\tau$-tilting subcategories in $\widehat{\A}$.
\end{thm}

Our fifth main result shows that under the settings of last theorem and certain conditions, every almost complete support $\tau$-tilting pair is exactly contained in two different support $\tau$-tilting pairs.

\begin{thm}{\rm (see Theorem \ref{last} for details)}
Let $R$ be a commutative noetherian ring which is complete and local. Let $\widehat{\A}$ be an {\rm Ext}-finite abelian category over $R$ with enough projectives. Let $\mathcal P$ be the subcategory of projective objects. Assume that $(\widehat{\X},\E)$ is a $\tau$-rigid pair satisfying the following conditions:
\begin{itemize}
\item[(X0)] $\widehat{\X}$ is not support $\tau$-tilting;
\item[(X1)] $\widehat{\X}$ is contravariantly finite;
\item[(X2)] every projective object admits a left $\widehat{\X}$-approximation;
\item[(X3)] $\E=\{P\in \mathcal P~|~\Hom_{\widehat{\A}}(P,\widehat{\X})=0\}$.
\end{itemize}
Then it is contained in two support $\tau$-tilting pairs $(\widehat{\M},\E)$ and $(\widehat{\N}, \E)$ such that
$$\Fac\widehat{\M}=\Fac\widehat{\X},\hspace{2mm} \Fac\widehat{\N}={^\perp(\tau\widehat{\X})}\cap{{\E}^{\perp}} \hspace{1mm} \text{ and } \hspace{1mm} \Fac \widehat{\M}\subsetneq \Fac \widehat{\N}.$$
Moreover, if $(\widehat{\X},\E)$ is an almost complete support $\tau$-tilting pair, then $(\widehat{\M},\E)$ and $(\widehat{\N},\E)$ are the only support $\tau$-tilting pairs which contain $(\widehat{\X},\E)$.
\end{thm}

This paper is organized as follows:
In Section 2, we recall some concepts about (relative) cluster tilting theory, then we show several results concerning relative rigid subcategories, which will be extensively utilized in the sequel.
In Section 3, we reveal a structured framework for mutations of two-term $\R[1]$-rigid subcategories.
In Section 4, we show that any almost complete two-term weak $\R[1]$-cluster tilting subcategory has exactly two completions.
In Section 5, we investigate the connection between relative cluster tilting subcategories and $\tau$-tilting theory in functor categories.
In Section 6, by using the results in the previous sections, we study the $\tau$-tilting theory in abelian categories.
In Section 7, we provide several examples to explain our main results.
\vspace{1mm}

We end this section with some conventions.
Let $\C$ be a triangulated category with shift functor $[1]$ and $\R$ be a rigid subcategory of $\C$.
For two subcategories $\X$ and $\Y$ of $\C$, we denote by $\X\ast \Y$ the subcategory in $\C$ consisting of all objects $M\in \C$ admitting triangles $X\longrightarrow M \longrightarrow Y \longrightarrow X[1],$ with $X\in \X$ and $Y\in \Y$. We denote by $\X\cup\Y$ the smallest subcategory of $\C$ containing $\X$ and $\Y$. For any subcategories $\mathcal D$ of an additive category, we denote by $\add \mathcal D$ the subcategory consisting of objects which are direct sums of direct summands of objects in $\D$. Especially, when $\mathcal D=\{D\}$, we denote $\add\mathcal D$ by $\add D$.

\vspace{2mm}

\section{Preliminaries}
\setcounter{equation}{0}

The class of cluster tilting subcategories is a significant class in triangulated categories, boasting a plethora of desirable properties. Let's delve into the definition of cluster tilting subcategories and related subcategories, as outlined in seminal works such as \cite{BMRRT,KR,KZ,IY}.

For convenience, we introduce the following notions.

\begin{defn}\label{def1}
Let $\C$ be a triangulated category with shift functor $[1]$.
\begin{enumerate}
\item[{\rm (i)}] A subcategory $\R$ of $\C$ is called {\rm rigid} if ${\rm Hom}_{\C}(\R, \R[1])=0$.
\item[{\rm (ii)}] A subcategory $\R$ of $\C$ is called {\rm maximal rigid} if it is rigid and maximal with respect to the property: If $M\in \C$  satisfies
     ${\rm Hom}_{\C}(\R\cup \add M, (\R\cup \add M)[1])=0$, then $M\in \R$.
\item[{\rm (iii)}] A functorially finite subcategory $\R$ of $\C$ is called cluster tilting if
$$\R=\{M\in\C\ |\ {\rm Hom}_{\C}(\R, M[1])=0\}=\{M\in\C\ |\ {\rm Hom}_{\C}(M, \R[1])=0\}.$$
\item[{\rm (iv)}] An object $R$ in $\C$ is called rigid, maximal rigid, or cluster tilting if $\add R$ is rigid, maximal rigid, or cluster tilting respectively.
\end{enumerate}
\end{defn}

For any subcategory $\mathcal D\subseteq\C$, denote by $[\mathcal D](A,B)$ the subset of ${\rm Hom}_{\C}(A,B)$ consisting of the morphisms that factors through objects in $\mathcal D$. Now we recall the notion of two-term relative cluster tilting subcategories and related notations from \cite{YZ,YZZ,ZhZ}.

\begin{defn}\label{def2}
Let $\R$ be a rigid subcategory of\hspace{1.5mm}$\C$.
\begin{itemize}

\item[(i)] A subcategory $\X$ in $\C$ is called  \emph{$\R[1]$-rigid }if  $[\R[1]](\X, \X[1])=0$. $\X$ is called \emph{two-term} $\R[1]$-rigid if in addition, $\X\subseteq \RR$.

\item[{(ii)}]  A subcategory $\X$ is called \emph{two-term maximal $\R[1]$-rigid if $\X\subseteq \RR$} is $\R[1]$-rigid and for any $M\in \RR$,
$$[\R[1]](\X\cup\add M, (\X\cup\add M)[1])=0 \text{ implies } M\in\X. $$

\item[{(iii)}]  A subcategory $\X$ is called {\rm two-term weak $\R[1]$-cluster tilting
 if $\X\subseteq \RR$,} $\R\subseteq\X[-1]\ast\X$ and $$\X=\{M\in\RR\ |\ [\R[1]](M, \X[1])=0 \;\, \emph{and } \ [\R[1]](\X, M[1])=0 \ \}.$$

\item[{(iv)}] A subcategory $\X$ is called {\rm two-term $\R[1]$-cluster tilting} if $\X$ is contravariantly finite and two-term weak $\R[1]$-cluster tilting.

\item[{(v)}] An object $X$ is called two-term $\R[1]$-rigid, two-term maximal $\R[1]$-rigid, two-term weak $\R[1]$-cluster tilting, or two-term $\R[1]$-cluster tilting if $\add X$ is two-term $\R[1]$-rigid, two-term maximal $\R[1]$-rigid, two-term weak $\R[1]$-cluster tilting, or two-term $\R[1]$-cluster tilting respectively.
\end{itemize}
\end{defn}

\begin{rem}\label{r0}
By definition, any two-term weak $\R[1]$-cluster tilting subcategory is a maximal $\R[1]$-rigid subcategory.
\end{rem}

From now on, let $\C$ be a Krull-Schmidt triangulated category with shift functor $[1]$. We assume that any subcategory we mentioned is full
and closed under isomorphisms. Let $\R$ be a rigid subcategory of $\C$, which is closed under direct sums and direct summands.

By definition, if $\U$ is an $\R[1]$-rigid subcategory, so is $\add \U$. Hence, in this paper, when we say $\U$ is an $\R[1]$-rigid subcategory, we always assume that $\U$ is closed under direct sums and direct summands.

\begin{rem}\label{rem0}
If $\C$ is $2$-Calabi-Yau, by {\rm \cite[Lemma 3.5]{YZZ}}, a subcategory of $\C\subseteq \R\ast \R[1]$ is two-term $\R[1]$-rigid if and only if it is rigid. If $\R$ is a cluster tilting subcategory, then $\RR=\C$. In this case, all subcategories can be viewed as ``two-term". Then by definition, a functorially finite subcategory $\X$ is cluster tilting if and only if it is two-term $\R[1]$-cluster tilting.
\end{rem}

%\begin{rem}\label{rem1}
% If $\R$ is a silting subcategory of $\C$, by {\rm \cite[Theorem 5.4]{ZhZ}},
%  we know that $\X$ is a two-term weak $\R[1]$-cluster tilting subcategory of
%  $\C$ if and only if $\X$ is a two-term silting subcategory of $\C$.
%\end{rem}

\begin{rem}
By definition, we know that any rigid subcategory is $\R[1]$-rigid.
\end{rem}

\begin{lem}\label{sum}
$\R*\R[1]$ is closed under direct sums, direct summands and isomorphisms.
\end{lem}

\begin{proof}
By definition, $\R*\R[1]$ is closed under direct sums and isomorphisms.
Since $\R$ is rigid, by \cite[Proposition 2.1]{IY}, we have that
$\R*\R[1]$ is closed under direct summands.
\end{proof}

For convenience, we introduce the following notion: for any subcategory $\X$, denote
$$ \{ R\in \R ~|~{\rm Hom}_{\C}(R,\X)=0\}$$
by $\R(\X)$.

\begin{lem}\label{lem0}
Let $\X$ be two-term $\R[1]$-rigid. If $\R'\subseteq \R(\X)$, then $\add(\X\cup\R'[1])$ is still two-term $\R[1]$-rigid.
\end{lem}

\begin{proof}
Since $\R$ is rigid, we have $[\R[1]](\X,\R'[2])=0$ and $\Hom_{\C}(\R'[1],\R'[2])=0$. Since $\R'\subseteq \R(\X)$, we have $0=\Hom_{\C}(\R'[1],\X[1])=[\R[1]](\R'[1],\X[1])$. By definition, $\add(\X\cup\R'[1])$ is still $\R[1]$-rigid. Since $\X\subseteq \R\ast \R[1]$, $\R'[1]\subseteq \R*\R[1]$, by Lemma \ref{sum}, $\add (\X\cup\R'[1])\subseteq \R\ast \R[1]$.
\end{proof}

On the other hand, we have the following lemma.

\begin{lem}\label{cor0}
Let $\X$ be $\R[1]$-rigid. Then $\X\cap \R[1]\subseteq  \R(\X)[1]$. Moreover, if $\X$ is two-term maximal $\R[1]$-rigid, then $\X\cap \R[1]=\R(\X)[1]$.
\end{lem}

\begin{proof}
Let $R[1]\in \X\cap \R[1]$. Since $\X$ is $\R[1]$-rigid, we have
$$0=[\R[1]](R[1],\X[1])={\rm Hom}_{\C}(R[1],\X[1]).$$
Hence ${\rm Hom}_{\C}(R,\X)=0$ and $\X\cap \R[1]\subseteq  \R(\X)[1]$. The ``Moreover" part is followed by the argument above and Lemma \ref{lem0}.
\end{proof}

\begin{lem}\label{R1eq}
Let $\X$ be a two-term $\R[1]$-rigid subcategory. If $\X\cap \R[1]=\R(\X)[1]$, then for any two-term maximal $\R[1]$-rigid subcategory $\M\supseteq \X$, we have $\R(\X)=\R(\M)$.
\end{lem}

\begin{proof}
By Lemma \ref{cor0}, we have
$$\X\cap \R[1]=\R(\X)[1]\supseteq \R(\M)[1]=\M\cap \R[1]\supseteq \X\cap \R[1].$$
\end{proof}

\begin{lem}\label{tau1}
Let $\U,\V$ be two subcategories of $\R\ast \R[1]$. Then $\R[1](\U,\V[1])=0$ if and only if any object $U\in \U$ admits a triangle
$$R_1\xrightarrow{f} R_2\xrightarrow{u_1} U\xrightarrow{u_2} R_1[1]$$
where $\Hom_{\C}(f,V)$ is surjective for any object $V\in \V$.
\end{lem}

\begin{proof}
Let $U$ be any object in $\U$. $U$ admits a triangle
$$U[-1]\xrightarrow{-u_2[-1]} R_1\xrightarrow{f} R_2\xrightarrow{u_1} U$$
with $R_1,R_2\in \R$.

Assume $\R[1](\U,\V[1])=0$. Let $r_1: R_1\to V$ be any morphism with $V\in \V$.
Since $$\R[1](\U,\V[1])=0,$$ we have $r_1(-u_2[-1])=0$, which implies $r_1$ factors through $f$. Hence $\Hom_{\C}(f,V)$ is surjective for any object $V\in \V$.

Now assume that $\Hom_{\C}(f,V)$ is surjective for any object $V\in \V$. Then in the triangle
$ R_2\xrightarrow{u_1} U\xrightarrow{u_2} R_1[1]\xrightarrow{-f[1]} R_2[1],$
$\Hom_{\C}(-f[1],\V[1])$ is surjective for any object $V\in \V$. Let $U\xrightarrow{a_1} R[1]\xrightarrow{a_2} V[1]$ be a morphism with $R\in \R$. Since ${\rm Hom}_\C(R_2,R[1])=0$, there is a morphism $b_1:R_1[1]\to R[1]$ such that $a_1=b_1u_2$. Then there is a morphism $b_2:R_2[1]\to V[1]$ such that $b_2(-f[1])=a_2b_1$. Hence $a_2a_1=b_2(-f[1])u_2=0$, which means  $\R[1](U,V[1])=0$.
\end{proof}

%\section{Relative rigid subcategories and related subcategories}
%In this section, we show some results about relative rigid subcategories, which will be frequently used in subsequent sections.

\begin{lem}\label{r1}
In the triangle
$$R \to A \to B\to R[1]$$
with $R\in \R$, $A\in \RR$ if and only if $B\in \RR$.
\end{lem}

\begin{proof}
Since $\R$ is rigid, we have $\R*\R=\R$. If $A\in \RR$, then $B\in (\RR)*\R[1]=\R*(\R[1]*\R[1])=\RR$. Conversely, if $B\in \RR$, then $A\in \R*(\RR)=(\R*\R)*\R[1]=\RR$.
\end{proof}

\begin{lem}\label{lem1}
Let
$$A\overset{x}{\longrightarrow}B\overset{y}{\longrightarrow}C\overset{~z~}{\longrightarrow}A[1]$$
be a triangle in $\C$. Then the following are equivalent.
\begin{itemize}
\item[\rm (1)] $x$ is left minimal;
\item[\rm (2)] $y$ is in the Jacobson radical;
\item[\rm (3)] $z$ is right minimal.
\end{itemize}
\end{lem}

\proof We first show that (1)$\Leftrightarrow$(2).

Assume that $y$ is in the Jacobson radical. Let $u\colon B \to B $ be a morphism
such that $ux=x$. Then we have $(u-1_{B})x=0$.
So there exists a morphism $v\colon C\to B$
such that $u-1_{B}=vy$.
$$\xymatrix{
A \ar[r]^x\ar@{.>}[dr]_0 &B\ar[r]^y\ar[d]^{u-1_B} &C\ar[r]^{z\;\;\;}
\ar@/^1pc/[dl]^{v}&A[1]\\
&B&
}
$$
Because $y$ is in the Jacobson radical, we obtain that
$u=1_B+vy$ is invertible. This shows that $x$ is left minimal.

Conversely, assume that $x$ is left minimal.
Let $b\colon C\to B$ be any morphism.
Put $c:=1_B+by$. It follows that
$cx=(1_B+by)x=x+byx=x$. Since $x$ is left minimal, $c$ is an isomorphism. This shows that
$y$ is in the Jacobson radical.

We show that (1)$\Rightarrow$(3).

Suppose that $x$ is left minimal.
For any morphism $b\colon C\to C$ such that $zb=z$,
we have the following commutative diagram
$$\xymatrix{
A\ar[r]^x\ar@{=}[d]&B \ar[r]^{y} \ar@{.>}[d]^{a} &C\ar[r]^{z\;\;} \ar[d]^{b} &A[1] \ar@{=}[d]\\
A \ar[r]^x &B \ar[r]^{y} &C \ar[r]^{z\;\;}&A[1].
}
$$
Then there exists a morphism
$a\colon B\to B$ such that $ax=x$.
Since $x$ is left minimal, we obtain that $a$ is isomorphism.
It follows that $b$ is also isomorphism.
This shows that $z$ is right minimal.

Dually we can show (3)$\Rightarrow$(1).
\qed

\begin{lem}\label{lem3}
Let $\X$ be an $\R[1]$-rigid subcategory and $L$ be any object in $\C$. Assume that $L$ admits a triangle
$$\xymatrix{L\ar[r]^{x_1}&X_1\ar[r]&X_0\ar[r]^{x_0~~} &L[1]}$$
where $X_1,X_0\in \X$.
\begin{itemize}
\item[\rm (1)] If $x_0$ factors through $\R[1]$, then $x_1$ is a left $\X$-approximation. Moreover, $L$ admits a triangle:
$$\xymatrix{L \ar[r]^{x_2}&X_2\ar[r]&X_4\ar[r]^{x_4~} &L[1]}$$
where $X_2,X_4\in\X$, $x_2$ is a minimal left $\X$-approximation and $x_4$ factors through $\R[1]$.
\item[\rm (2)] If $x_1$ factors through $\R$, then $x_0$ is a right $\X$-approximation. Moreover, $L$ admits a triangle:
$$\xymatrix{L \ar[r]^{x_2}&X_2\ar[r]&X_4\ar[r]^{x_4~} &L[1]}$$
where $X_2,X_4\in\X$, $x_4$ is a minimal right $\X$-approximation and $x_2$ factors through $\R$.
\end{itemize}
\end{lem}

\proof We only show (1), the proof of (2) is similar.
\vspace{1mm}

Since $\X$ is $\R[1]$-rigid and $-x_0[-1]$ factors through $\R$, from the triangle
$$\xymatrix{X_0[-1]\ar[r]^-{-x_0[-1]}&L\ar[r]^{x_1}&X_1\ar[r]&X_0}$$
we get that any morphism for $L$ to $\X$ factors through $x_1$, hence $x_1$ is a left $\X$-approximation.

Since $\C$ is Krull-Schmidt, there exists a decomposition $X_1=X_2\oplus X_3$ such that
$$x_1=\binom{x_2}{0}\colon L\to X_2\oplus X_3,$$
where $X_2\in \X$, $x_2$ is left minimal.
Since $x_1$ is a left $\X$-approximation of $R$, so is $x_2$.
Then there exists a triangle
\begin{equation}\label{ttt2}
\begin{array}{l}
\xymatrix{L\ar[r]^{x_2}&X_2\ar[r]&X_4\ar[r]^{x_4}&L[1].}
\end{array}
\end{equation}
It follows that
$X_0\cong X_4\oplus X_3$ which implies $X_4\in\X$. Thus the triangle (\ref{ttt2}) is what we expect.  \qed

\begin{lem}\label{lem2}
Let $\X$ be an  $\R[1]$-rigid subcategory and $L$ be an $\R[1]$-rigid object. Given a triangle
\begin{equation}\label{ttt1}
\begin{array}{l}
\xymatrix{L\ar[r]^{x}&X_1\ar[r]^y&X_0\ar[r]^{z\;\;\;}&L[1]}
\end{array}
\end{equation}
where $X_0,X_1\in\X$, $x$ is left minimal and $z$ factors through $\R[1]$. Then $X_0$ and $X_1$ do not have common indecomposable direct summands.
\end{lem}

\proof  Let $f\colon X_1\to X_0$ be any morphism. It is enough to show that $f$ is in the Jacobson radical.
In fact, applying the functor $\Hom_{\C}(L,-)$ to
the triangle (\ref{ttt1})
we have the following exact sequence:
$$\Hom_{\C}(L,X_1)\xrightarrow{{\small \rm Hom_{\C}}(L,\hspace{0.3mm}y)}\Hom_{\C}(L,X_0)
\xrightarrow{{\small \rm Hom_{\C}}(L,\hspace{0.3mm}z)=0}{\rm Hom}_{\C}(L,L[1]).$$
So there exists a morphism
$g\colon L\to X_1$ such that $yg=fx$.
By Lemma \ref{lem3}, $x$ is a left $\X$-approximation, so
there exists a morphism $h\colon X_1\to X_1$
such that $hx=g$.
It follows that $yhx=fx$ and then $(f-yh)x=0$.
So there exists a morphism $w\colon X_0\to X_0$
such that $f-yh=wy$.
$$\xymatrix{
L\ar[r]^x\ar@{.>}[dr]_0 &X_1\ar[r]^y\ar[d]^{f-yh}&X_0\ar[r]^{z\;\;\;}\ar@/^1pc/[dl]^{w}&L[1]\\
&X_0&
}
$$
Since $x$ is left minimal, by Lemma \ref{lem1}, we have that $y$ is in the Jacobson radical.
This shows that $f=yh+wy$ is in the Jacobson radical, hence any morphism from $X_1$ to $X_0$ is in the Jacobson radical.
%
%Suppose that
%$X_0$ and $X_1$ have a common indecomposable direct summand $X_2$, we have decompositions
%$X_0=X_2\oplus X_0'$ and $X_1=X_2\oplus X_1'$.
%But the composition
%$\left(\begin{smallmatrix}1&0\\[1mm]
%0&0\end{smallmatrix}\right)\colon X_1\xrightarrow{~(1,\hspace{0.3mm}0)~}X_2
%\xrightarrow{~\binom{1}{0}~}X_0$
%is not in the Jacobson radical. This is a contradiction.
\qed

\begin{lem}\label{lem4}
Let $\X$ be a two-term $\R[1]$-rigid subcategory. Then the following are equivalent.
\begin{itemize}
\item[\rm (1)] $\X$ is two-term maximal $\R[1]$-rigid such that each object $R\in \R$ admits a left $\X$-approximation.
\item[\rm (2)] $\R\subseteq \X[-1]*\X$.
\item[\rm (3)] $\X$ is two-term weak $\R[1]$-cluster tilting.
\end{itemize}
\end{lem}

\proof
(1)$\Rightarrow$(2): this is followed by \cite[Lemma 2.6(1)]{ZhZ}.

(3)$\Rightarrow$(1): this is followed by definition and Lemma \ref{lem3}(1).

(2)$\Rightarrow$(3):  Since $\X\subseteq \{M\in\RR\ |\ [\R[1]](M, \X[1])=0 \;\, \emph{and } \ [\R[1]](\X, M[1])=0 \ \}$, we only need to show that if $M\in \RR$ satisfies
$[\R[1]](\X,M[1])=0$ and $[\R[1]](M, \X[1])=0$, then $M\in \X$.

Since $M\in\RR$,
there exists a triangle
$$\xymatrix{R_1\ar[r]^{f}&R_0\ar[r]^{g}&M\ar[r]^{h\;\;}&R_1[1]}$$
where $R_0,R_1\in\R$. Note that $R_0\in\R\subseteq\X[-1]\ast\X$, there exists a triangle
$$\xymatrix{R_0\ar[r]^{u}&X_1\ar[r]^{v}&X_2\ar[r]^{w\;\;}&R_0[1]}$$
with $X_1,X_2\in\X$. Since $\X$ is $\R[1]$-rigid, $u$ is a left $\X$-approximation of $R_0$ by Lemma \ref{lem3}.
Then we have a commutative diagram.
$$\xymatrix{
R_1\ar[r]^{f}\ar@{=}[d]&R_0\ar[r]^{g}\ar[d]^{u}&M\ar[r]^{h}\ar[d]^{a}&R_1[1]\ar@{=}[d]\\
R_1\ar[r]^{x=uf}&X_1\ar[r]^{y}\ar[d]^{v}&N\ar[r]^{z}\ar[d]^{b}&R_1[1]\\
&X_2\ar@{=}[r]\ar[d]^{w}&X_2\ar[d]^{c}\\
&R_0[1]\ar[r]^{g[1]}&M[1]}$$
Since $c\in [\R[1]](X_2,M[1])=0$, we have $N\cong X_2\oplus M$. By Lemma \ref{tau1}, $\Hom_\C(f,X)$ is surjective for any object $X\in\X$. Since $u$ is a left $\X$-approximation, $x=uf$ is a left $\X$-approximation of $R_1$. $R_1$ admits a triangle
$$R_1\xrightarrow{r_1} X_1'\to X_2'\to R[1]$$
with $X_1',X_2'\in \X$. By Lemma \ref{lem3}, $r_1$ is a left $\X$-approximation. Hence we have the following commutative diagram:
$$\xymatrix{
R_1 \ar@{=}[d] \ar[r]^{x} &X_1\ar[r]^{y}\ar[d] &N \ar[r]^{z}\ar[d] &R_1[1] \ar@{=}[d]\\
R_1 \ar@{=}[d] \ar[r]^{r_1} &X_1' \ar[r] \ar[d]  &X_2' \ar[r] \ar[d] &R_1[1] \ar@{=}[d]\\
R_1\ar[r]^{x} &X_1 \ar[r]^{y} &N \ar[r]^{z} &R_1[1].\\
}
$$
Then $N$ is a direct summand of $X_1\oplus X_2'\in \X$, which implies that $M\in \X$.
\qed

\begin{cor}\label{cor3}
Let $\X\subseteq \RR$ be a subcategory such that each object in $\R[1]$ admits a right $\X$-approximation. Then it is two-term maximal $\R[1]$-rigid if and only if it is two-term weak $\R[1]$-cluster tilting.
\end{cor}

\begin{proof}
By definition, we know that any two-term  weak $\R[1]$-cluster tilting subcategory is two-term maximal $\R[1]$-rigid. If $\X$ is two-term maximal $\R[1]$-rigid such that each object in $\R[1]$ admits a right $\X$-approximation, by \cite[Corollary 2.7]{ZhZ}, $\R\subseteq \X[-1]*\X$. Then by Lemma \ref{lem4}, $\X$ is two-term weak $\R[1]$-cluster tilting.
\end{proof}

\section{Mutations of relative cluster tilting subcategories}
%In this section, we give an analogue of the co-Bongartz completion and elucidate a structured framework for irreducible mutation within two-term $\R[1]$-rigid subcategories.

\begin{prop}\label{facmax}
Let $\U$ be a two-term $\R[1]$-rigid subcategory such that each object in $\R$ admits a left $\U$-approximation. Then it is contained in a two-term weak $\R[1]$-cluster tilting subcategory.
\end{prop}

\begin{proof}
Let $R$ be any object in $\R$. $R$ admits a triangle
$$R\xrightarrow{f} U^R\xrightarrow{g} V^R\xrightarrow{h} R[1]$$
where $f$ is a left $\U$-approximation. Let $\M=\add(\U\cup\{V^R~|~R\in \R\})$. By Lemma \ref{lem4}(2), we only need to show that $\M$ is two-term $\R[1]$-rigid.

{\bf Step 1:} By Lemma \ref{r1} and Lemma \ref{sum}, $\M\subseteq \R*\R[1]$.

{\bf Step 2:} We show that $[\R[1]](V^R,\U[1])=0$.

Let $U\in\U$ and $u\in [\R[1]](V^R,U[1])$. Since $\U$ is $\R[1]$-rigid, $ug\in [\R[1]](U^R,U[1])=0$. Then there is a morphism $v:R[1]\to U[1]$ such that $u=vh$. Since $f$ is a left $\U$-approximation, $-f[1]$ is a left $\U[1]$-approximation. Then we have the following commutative diagram
$$\xymatrix{U^R \ar[r]^g &V^R \ar[d]_u \ar[r]^h &R[1] \ar@{.>}[ld]_{v} \ar[r]^-{-f[1]} &U^R[1] \ar@{.>}[dll]^w\\
&U[1]
}
$$
which implies that $u=vh=w(-f[1])h=0$.

{\bf Step 3:} We show that $[\R](\U[-1],V^R)=0$.

Let $U\in \U$ and $u\in [\R](U[-1],V^R)=0$. Then $u$ can be written as $u:U[-1]\xrightarrow{u_1} R_0\xrightarrow{u_2} V^R$ where $R_0\in \R$. Since $\R$ is rigid, $hu_2=0$. Then we have the following commutative diagram
$$\xymatrix{
U[-1] \ar[r]^{u_1} &R_0 \ar[d]^{u_2} \ar@{.>}[dl]_{r_0}\\
U^R \ar[r]_g &V^R \ar[r]_h &R[1].
}
$$
Since $\U$ is $\R[1]$-rigid, $r_0u_1=0$, which implies that $u=u_2u_1=gr_0u_1=0$.

{\bf Step 4:} By using the same method as in Step 3 and the result of Step 3, we can get that $[\R[1]](V^R_1,V^R_2[1])=0$ for any $V^R_1,V^R_2\in \{V^R~|~R\in \R\}$. Hence $\M$ is $\R[1]$-rigid.
\end{proof}

For convenience, we denote the two-term weak $\R[1]$-cluster tilting subcategory which contains $\U$ in Proposition \ref{facmax} by $\M_{\U}$.

\begin{prop}\label{NX}
By {\rm\cite[Theorem 3.1]{ZhZ}}, for any two-term $\R[1]$-rigid subcategory $\U$ such that each object in $\R[1]$ admits a right $\U$-apporximation, we can construct a two-term weak $\R[1]$-cluster tilting subcategory $\mathcal N_{\U}$ which contains $\U$ in the following way:

For any object $R\in \R$, take a triangle
$$R\to V_R\to U_R\xrightarrow{~h~} R[1]$$
where $h$ is a right $\U$-approximation of $R[1]$. Then
$$\mathcal N_{\U}=:\add(\U\cup\{V_R ~|~ R\in \R\})$$
is the two-term weak $\R[1]$-cluster tilting subcategory we mentioned above.
\end{prop}

Let $\X$ be a subcategory of $\C$. For convenience, if each object in $\R$ admits a left $\X$-approximation and each object in $\R[1]$ admits a right $\X$-approximation, we say $\X$ is $\R[1]$-functorially finite.

\begin{rem}
By Lemma \ref{lem3}, any two-term weak $\R[1]$-cluster tilting subcategory is $\R[1]$-functorially finite.
\end{rem}

Let $\X$ and $\Y$ be two subcategories of $\C$. We denote by  $\Y\backslash\X$ the full
subcategory of $\C$ consisting of all objects $Y\in \Y$ such that $Y\notin \X$.

\begin{lem}\label{cap}
Let $\X$ be an $\R[1]$-functorially finite two-term $\R[1]$-rigid subcategory.
\begin{itemize}
\item[\rm (1)] $\R(\M_\X)=\R(\X)$;
\item[\rm (2)] $\mathcal N_\X\cap \R[1]=\X\cap \R[1]$.
\end{itemize}
\end{lem}

\begin{proof}
(1) We only need to show that $\R(\M_\X)\supseteq \R(\X)$.

Let $R_0\in \R(\X)$ and $M_0$ be an indecomposable object in $\M_\X\backslash \X$. $M_0$ admits a triangle
$$R\xrightarrow{r} X\xrightarrow{x} M\xrightarrow{m} R[1]$$
where $r$ is a left $\X$-approximation and $M_0$ is a direct summand of $M$. For any morphism $r_0:R_0\to M$, since $mr_0=0$, there is a morphism $x_0:R_0\to X$ such that $xx_0=r_0$. Since $x_0=0$, we have $r_0=0$. Hence $\R(\M_\X)\supseteq \R(\X)$.

(2) We only need to show that $\mathcal N_\X\cap \R[1]\subseteq \X\cap \R[1]$.

Let $N_0$ be an indecomposable object in $\mathcal N_\X\cap \R[1]$. If $N_0\notin \X$, then it admits a triangle
$$R\to N\to X\xrightarrow{x} R[1]$$
where $x$ is a right $\X$-approximation and $N_0$ is direct summand of $N$. Since ${\rm Hom}_{\C}(R,N_0)=0$, $N_0$ is a direct summand of $X$, a contradiction.
\end{proof}

%\begin{lem}\label{alter}
%Let $\X$ be an $\R[1]$-functorially finite two-term $\R[1]$-rigid subcategory. Let $R\in \R$. Then $R$ admits a triangle
%$$R\xrightarrow{a} X_1\to M_1\xrightarrow{b} R[1]$$
%where $X_1\in \X$ and $M_1\in \M_\X \backslash \X$ if and only if it admits a triangle
%$$R\xrightarrow{a'} N_1\to X_1'\xrightarrow{b'} R[1]$$
%where $X_1'\in \X$ and $N_1\in \mathcal N_\X \backslash \X$.
%\end{lem}
%
%\begin{proof}
%Assume that $R$ admits a triangle
%$$R\xrightarrow{a} X_1\to M_1\xrightarrow{b} R[1]$$
%where $X_1\in \X$ and $M_1\in \M_\X \backslash \X$. Note that $R$ also admits a triangle
%$$R\xrightarrow{a'} N_1\to X_1'\xrightarrow{b'} R[1]$$
%where $b'$ is a right $\X$-approximation and $N_1\in \mathcal N_\X$. By Lemma \ref{lem3}, $b$ is a right $\M_\X$-approximation. If $N_1\in \X$, then $b'$ is also a right $\M_\X$-approximation. We can get the following commutative diagram:
%$$\xymatrix{
%R \ar[r] \ar@{=}[d] &X_1 \ar[r] \ar[d] &M_1 \ar[r]^b \ar[d] &R[1] \ar@{=}[d] \\
%R \ar[r] \ar@{=}[d] &N_1 \ar[r] \ar[d] &X_1' \ar[r]^{b'} \ar[d] &R[1] \ar@{=}[d] \\
%R \ar[r]  &X_1 \ar[r]  &M_1 \ar[r]^b &R[1]
%}
%$$
%which implies $M_1$ is a direct summand of $X_1\oplus X_1'$, a contradiction. By the similar method we can show the ``if" part.
%\end{proof}

\begin{lem}\label{pre-mu}
Let $\X$ be an $\R[1]$-functorially finite two-term $\R[1]$-rigid subcategory. Then
\begin{itemize}
\item[\rm (1)] any object in $\M_\X$ admits a right $\X$-approximation;
\item[\rm (2)] any object in $\N_\X$ admits a left $\X$-approximation.
\end{itemize}
\end{lem}

\begin{proof}
(1)~ Let $M$ be any indecomposable object in $\M_\X\backslash \X$. Then $M$ admits a triangle
$$R\xrightarrow{a} X_1\to M_1\xrightarrow{b} R[1]$$
where $X_1\in \X$, $M_1\in \M_\X$ and $M$ is a direct summand of $M_1$. $R$ also admits a triangle
$$R\xrightarrow{a'} N_1\to X_1'\xrightarrow{b'} R[1]$$
where $X_1'\in \X$ and $N_1\in \mathcal N_\X$. Then we have the following commutative diagram
$$\xymatrix{
&N_1 \ar@{=}[r] \ar[d] &N_1\ar[d]\\
X_1 \ar[r] \ar@{=}[d] &X_1\oplus X_1' \ar[r] \ar[d]^{\chi} &X_1' \ar[r]^0 \ar[d]^{b'} &X_1[1] \ar@{=}[d]\\
X_1 \ar[r] &M_1 \ar[r] \ar[d] &R[1] \ar[r] \ar[d] &X_1[1]\\
&N_1[1] \ar@{=}[r] &N_1[1]
}
$$
By Lemma \ref{lem3}(2), $\chi$ is a right $\X$-approximation. Hence $M$ admits a right $\X$-approximation.

Dually we can prove (2).
\end{proof}

\begin{lem}\label{mutation}
Let $\X$ be an $\R[1]$-functorially finite two-term $\R[1]$-rigid subcategory. Let $Y\in \M_\X \backslash \X$ be an indecomposable object. Then in the following triangle
$$Z\xrightarrow{z} X\xrightarrow{x} Y\xrightarrow{y} Z[1]$$
where $x$ is a minimal right $\X$-approximation, we can obtain that:
\begin{itemize}
\item[\rm (1)] $y$ factors through $\R[1]$;
\item[\rm (2)] $z$ is a minimal left $\M_\X$-approximation;
\item[\rm (3)] $Z$ is indecomposable and $Z\notin \X$;
\item[\rm (4)] $[\R[1]](\X,Z[1])=0$.
\end{itemize}
\end{lem}

\begin{proof}
By Lemma \ref{pre-mu}, $Y$ always admits a triangle
$$Z\xrightarrow{z} X\xrightarrow{x} Y\xrightarrow{y} Z[1]$$
where $x$ is a minimal right $\X$-approximation.

(1) By Proposition \ref{facmax}, there exists a triangle
$$R\xrightarrow{f} X_1\xrightarrow{g} M_1\xrightarrow{h} R[1]$$
where $R\in \R$ and $f$ is a left $\X$-approximation such that $Y$ is a direct summand of $M_1$. Then we have a retraction $p\colon M_1\to Y$ and a section $i\colon Y\to M_1$ such that $pi=1_Y$. Since $x$ is a right $\X$-approximation, we have the following commutative diagram
$$\xymatrix{
R \ar[r]^f \ar[d] &X_1 \ar[r]^g \ar[d] &M_1 \ar[r]^h \ar[d]^p &R[1] \ar[d]^r \\
Z \ar[r]_z &X \ar[r]_x &Y \ar[r]_y &Z[1].
}
$$
Then $y=ypi=rhi$, which means $y$ factors through $\R[1]$.
\medskip

(2) By Lemma \ref{lem3}, $z$ is a left $\M_\X$-approximation. If $z$ is not left minimal, we can rewrite the triangle as
$$Z\xrightarrow{\svecv{z_1}{0}} X'\oplus X'' \xrightarrow{\svech{x_1}{x_2}} Y\xrightarrow{y} Z[1]$$
Then $X'\oplus X''\xrightarrow{\svech{0}{1}} X''$ factors through $\svech{x_1}{x_2}$, which implies that $X''\cong Y$, a contradiction.
\medskip

(3) Note that $Z\notin \X$, otherwise $y=0$, then $Y$ becomes a direct summand of $X$. Let $Z_0$ be an indecomposable direct summand of $Z$ such that $Z_0\notin \X$. Then we have a retraction $p_0:Z\to Z_0$ and a section $i_0:Z_0\to Z$ such that $p_0i_0=1_{Z_0}$. $Z_0$ admits a triangle
$$Z_0\xrightarrow{zi_0} X\to Y_0\to Z_0[1]$$
where $Y_0\neq 0$. Since $z$ is a minimal left $\M_\X$-approximation, we can obtain the following commutative diagram
$$\xymatrix{
Z_0 \ar[r]^{zi_0} \ar[d]_{i_0} &X \ar[r] \ar@{=}[d] &Y_0 \ar[r] \ar[d]^{b_1} &Z_0[1] \ar[d]\\
Z \ar[r]^z \ar[d]_{p_0} &X \ar[r]^x \ar[d]^{a_2} &Y \ar[r]^y \ar[d]^{b_2} &Z[1] \ar[d]\\
Z_0 \ar[r]^{zi_0} &X \ar[r] &Y_0 \ar[r] &Z_0[1].
}
$$
Since $zi_0$ is also left minimal, $a_2$ is an isomorphism. Hence $b_2b_1$ is also an isomorphism. But $Y$ is indecomposable, then $b_1$ is an isomorphism. Thus $i_0$ is an isomorphism and $Z$ is indecomposable.

(4) Let $X_0$ be any object in $\X$ and $\alpha$ be any morphism in $[\R[1]](X_0, Z[1])$. Since $[\R[1]](X_0,X[1])=0$, there is a morphism $c:X_0 \to Y$ such that $\alpha=yc$. Since $x$ is right $\X$-approximation, there is a morphism $d:X_0\to X$ such that $c=xd$. Hence $\alpha=yxc=0$, which implies that $[\R[1]](\X,Z[1])=0$.
\end{proof}

Dually, we have the following lemma.

\begin{lem}\label{mutation-D}
Let $\X$ be an $\R[1]$-functorially finite two-term $\R[1]$-rigid subcategory. Let $Z\in \mathcal N_\X \backslash \X$ be an indecomposable object. Then in the following triangle
$$Z\xrightarrow{z} X\xrightarrow{x} Y\xrightarrow{y} Z[1]$$
where $z$ is a minimal left $\X$-approximation, we can obtain that:
\begin{itemize}
\item[\rm (1)] $y$ factors through $\R[1]$;
\item[\rm (2)] $x$ is a minimal right $\mathcal N_\X$-approximation;
\item[\rm (3)] $Y$ is indecomposable and $Y\notin \X$;
\item[\rm (4)] $[\R[1]](Y,\X[1])=0$.
\end{itemize}
\end{lem}

\begin{defn}\label{defmu}
Let $\X$ be a two-term $\R[1]$-rigid subcategory. Let $\M\neq\N$ be two-term weak $\R[1]$-cluster tilting subcategories which contain $\X$. $(\M,\N)$ is called an $\X$-mutation pair if the following conditions are satisfied:

{\rm (a)} Any object $Y\in \M$ admits a triangle
$$Z\xrightarrow{z} X\xrightarrow{x} Y\xrightarrow{y} Z[1]$$
such that
\begin{itemize}
\item[\rm (a1)] $y$ factors through $\R[1]$;
\item[\rm (a2)] $X\in \X$;
\item[\rm (a3)] $Z\in \N$.
\end{itemize}
{\rm (b)} Any object $Z'\in \N$ admits a triangle
$$Z'\xrightarrow{z'} X'\xrightarrow{x'} Y'\xrightarrow{y'} Z'[1]$$
such that
\begin{itemize}
\item[\rm (b1)] $y'$ factors through $\R[1]$;
\item[\rm (b2)] $X'\in \X$;
\item[\rm (b3)] $Y'\in \M$.
\end{itemize}
We say that $\M$ (resp. $\N$) is a left (resp. right) mutation of $\N$ (resp. $\M$) if there exists a two-term $\R[1]$-rigid subcategory $\X_0$ such that $(\M,\N)$ is an $\X_0$-mutation pair.
\end{defn}

\begin{rem}\label{rf}
Since two-term weak $\R[1]$-cluster tilting subcategories are $\R[1]$-functorially finite, by Lemma \ref{pre-mu}, if $(\M,\N)$ is an $\X$-mutation pair, then $\X$ is $\R[1]$-functorially finite.
\end{rem}

\begin{thm}\label{m-pair}
Let $\X$ be an $\R[1]$-functorially finite two-term $\R[1]$-rigid subcategory. Then $(\M_\X,\N_\X)$ is an $\X$-mutation pair.
\end{thm}

\begin{proof}
Let $Y\in \M_\X \backslash \X$ be an indecomposable object. Then by Lemma \ref{mutation}, $Y$ admits a triangle
$$Z\xrightarrow{z} X\xrightarrow{x} Y\xrightarrow{y} Z[1]$$
where $x$ is a minimal right $\X$-approximation. We show that $Z\in \mathcal N_\X$.

$Y$ admits a triangle
$$R\xrightarrow{a} X_1\to M_1\xrightarrow{b} R[1]$$
where $R\in \R$, $a$ is a left $\X$-approximation and $Y$ is a direct summand of $M_1\in \M_\X$. Let $M_1=X_0\oplus \widetilde{Y}$ where $X_0\in \X$ and $\widetilde{Y}$ is the direct sum of all the indecomposable direct summands of $M_1$ in $\M_\X \backslash \X$. Then $Y$ is a direct summand of $\widetilde{Y}$. By Lemma \ref{mutation}, $\widetilde{Y}$ admits a triangle
$$\widetilde{Z}\xrightarrow{\widetilde{z}} \widetilde{X}\xrightarrow{\widetilde{x}} \widetilde{Y}\xrightarrow{\widetilde{y}} \widetilde{Z}[1]$$
where $\widetilde{X}\in \X$, $\widetilde{y}$ factors through $\R[1]$ and $Z$ is a direct summand of $\widetilde{Z}$. Then we have the following commutative diagram.
$$\xymatrix{
&\widetilde{Z} \ar@{=}[r] \ar[d] &\widetilde{Z} \ar[d]\\
R \ar[r]^r \ar@{=}[d] &N \ar[r] \ar[d] &X_0\oplus \widetilde{X} \ar[r]^{\chi} \ar[d]^{\left(\begin{smallmatrix}1&0\\0&\widetilde{x}\end{smallmatrix}\right)} &R[1] \ar@{=}[d]\\
R \ar[r] &X_1 \ar[r] \ar[d]_c &X_0\oplus \widetilde{Y} \ar[d]^{\svech{0}{\widetilde{y}}} \ar[r]^{b} & R[1]\\
&\widetilde{Z}[1] \ar@{=}[r] &\widetilde{Z}[1]
}
$$
Since $c$ factors through $\R[1]$, by Lemma \ref{mutation}(4), $c=0$. Then $N\cong \widetilde{Z}\oplus X_1$. By Lemma \ref{lem3}, $b$ is a right $\mathcal M_\X$-approximation. Let $x_2:X_2\to R[1]$ be any morphism with $X_2\in \X$. Then there is a morphism $\chi_2:X_2\to X_0\oplus \widetilde{Y}$ such that $b\chi_2=x_2$. Since $\widetilde{y}$ factors through $\R[1]$, by Lemma \ref{mutation}(4), $\chi_2$ factors through $\left(\begin{smallmatrix}1&0\\0&\widetilde{x}\end{smallmatrix}\right)$. Hence $x_2$ factors through $\chi$, which means $\chi$ is a right $\X$-approximation. Then $N\in \mathcal N_\X$ and thus $Z\in \mathcal N_\X$.

Dually we can show that if $Z\in \mathcal N_\X \backslash \X$ be an indecomposable object, then $Z$ admits a triangle
$$Z\xrightarrow{z} X\xrightarrow{x} Y\xrightarrow{y} Z[1]$$
where $z$ is a minimal left $\X$-approximation and $Y\in \M_\X \backslash \X$.
\end{proof}

\begin{cor}\label{capcap}
Let $\X$ be an $\R[1]$-functorially finite two-term $\R[1]$-rigid subcategory. Then $\M_\X\cap \mathcal N_\X=\X$.
\end{cor}

\begin{proof}
We only need to show that $\M_\X\cap \mathcal N_\X\subseteq \X$.

If there is an indecomposable object $L\in (\M_\X\cap \mathcal N_\X)\backslash \X$, by Theorem \ref{m-pair}, it admits a triangle
$$L'\to X\to L\xrightarrow{l} L'[1]$$
where $L'\in \mathcal N_\X$ and $l$ factors through $\R[1]$. Since $L\in \mathcal N_\X$, we have $l=0$ and $L$ becomes a direct summand of $X$. Then $L\in \X$, a contradiction. Hence $\M_\X\cap \mathcal N_\X\subseteq \X$.
\end{proof}

On the other hand, we have the following proposition.

\begin{prop}\label{m-pair-2}
Let $\X$ be a two-term $\R[1]$-rigid subcategory. If $(\M,\N)$ is an $\X$-mutation pair, then $\M=\M_\X$ and $\N=\N_\X$.
\end{prop}

\begin{proof}
By Remark \ref{rf}, $\X$ is $\R[1]$-functorially finite.

We only show $\M=\M_\X$, by duality, we can get $\N=\N_\X$. Let $M_0$ be any object in $\M$. Then $M_0$ admits a triangle
$$N_0\to X_0\xrightarrow{x_0} M_0\xrightarrow{m_0} N_0[1]$$
where $x_0$ is a right $\X$-approximation and $m_0$ factors through $\R[1]$. Let $Y_0$ be any indecomposable object in $\M_\X\backslash \X$. Then $Y_0$ admits a triangle
$R\xrightarrow{x} X\to M_X\xrightarrow{m} R[1]$
where $R\in \R$, $x$ is a left $\X$-approximation, $M_X\in \M_\X$ and $Y_0$ is a direct summand of $M_X$. Let $f:M_0[-1]\to R_0\xrightarrow{r} M_X$ be any morphism in $[\R](M_0[-1],M_X)$. Since $mr=0$, $r$ factors through $X$. Since $[\R](\M[-1],\X)=0$, we have $f=0$. Let $g:M_X[-1]\to R_0'\xrightarrow{r'} M_0$ be any morphism in $[\R](M_X[-1],M_0)$. Since $m_0r'=0$, $r'$ factors through $X_0$. Since $[\R](\M_\X[-1],\X)=0$, we have $g=0$. Hence $[\R](\M[-1],\M_\X)=0$ and $[\R](\M_\X[-1],\M)=0$, which implies $\M= \M_\X$.
\end{proof}

%We call a subcategory $\mathcal S$ an $m$-rigid subcategory if $\Hom_\C(\mathcal S,\mathcal S[i])=0$, $i=1,2,\cdots,m$. $\mathcal S$ is called a presilting subcategory if $\mathcal S$ is $m$-rigid, $\forall m>0$; $\mathcal S$ is called a silting subcategory if it is presilting and $\C={\rm thick}\mathcal S$. In the rest of this section, we assume that $\R$ is $2$-rigid. From now on, we denote $\RR$ by $\A$.

%\begin{lem}\label{pres}
%If $\R$ is presilting, then $\Hom_{\C}(X,Y[i])=0$ for any $X,Y\in \A$ and $i\geq 2$.
%\end{lem}
%
%\begin{proof}
%Let $f\in \Hom_{\C}(X,Y[i])=0$. Consider the following diagram
%$$\xymatrix{
%
%}
%$$
%\end{proof}

\vspace{2mm}

\section{Almost complete relative cluster tilting subcategories}
In this section, we give the definition of almost complete two-term weak $\R[1]$-cluster tilting subcategories, and prove that such subcategories
has exactly (up to isomorphism) two completions.

\begin{defn}\label{almost}
We call a two-term $\R[1]$-rigid subcategory $\X$ an almost complete two-term weak $\R[1]$-cluster tilting subcategory if the following conditions are satisfied:
\begin{itemize}
\vspace{1mm}
\item[\rm (a)] $\X$ is $\R[1]$-functorially finite;
\vspace{1mm}

\item[\rm (b)] There exists an indecomposable object $W\notin \X$ such that $\add(\X\cup\{W\})$ is two-term weak $\R[1]$-cluster tilting.
\end{itemize}
Any two-term weak $\R[1]$-cluster tilting subcategory which has the form in {\rm (b)} is called a completion of $\X$, $W$ is called a complement of $\X$.
\end{defn}

For convenience, if $\X$ is an almost complete two-term weak $\R[1]$-cluster tilting subcategory, we can just say $\X$ is almost complete.

\begin{lem}\label{neq}
Let $\X$ be an $\R[1]$-functorially finite two-term $\R[1]$-rigid subcategory. Assume $\M_\X=\add(\X\cup\{Y\})$ where $Y\notin \X$ is indecomposable. By Lemma \ref{mutation} and Theorem \ref{m-pair}, $Y$ admits a triangle
$$Z\xrightarrow{z} X\xrightarrow{x} Y\xrightarrow{y} Z[1]$$
where $x$ is a minimal right $\X$-approximation and $Z\in \mathcal N_\X\backslash \X$ is indecomposable. Then $\mathcal N_\X=\add(\X\cup\{Z\})$.
\end{lem}

\begin{proof}
Let $Z_0$ be an indecomposable object in $\mathcal N_\X \backslash \X$. By Theorem \ref{m-pair}, $Z_0$ admits a triangle
$$Z_0\xrightarrow{z_0} X_0\xrightarrow{x_0} Y_0\xrightarrow{y_0} Z_0[1]$$
where $z_0$ is a left minimal right $\X$-approximation and $Y_0\in \M_\X\backslash \X$ is indecomposable. Then $Y_0$ has to be isomorphic to $Y$. Hence $Z_0\cong Z$.
\end{proof}

Dually, we have the following lemma.

\begin{lem}\label{neq-D}
Let $\X$ be an $\R[1]$-functorially finite two-term $\R[1]$-rigid subcategory. Assume $\mathcal N_\X=\add(\X\cup\{Z\})$ where $Z\notin \X$ is indecomposable. By Lemma \ref{mutation-D} and Theorem \ref{m-pair}, $Z$ admits a triangle
$$Z\xrightarrow{z} X\xrightarrow{x} Y\xrightarrow{y} Z[1]$$
where $z$ is a minimal left $\X$-approximation of $Z$ and $Y\in \M_\X\backslash \X$ is indecomposable. Then $\M_\X=\add(\X\cup\{Y\})$.
\end{lem}

The following result shows that any almost complete two-term weak $\R[1]$-cluster tilting subcategory has exactly two completions.

\begin{thm}\label{main1}
Let $\X$ be an almost complete two-term weak $\R[1]$-cluster tilting subcategory. Then $\M_\X$ and $\N_\X$ are completions of $\X$. Moreover, if $\U$ is a two-term weak $\R[1]$-cluster tilting subcategory which contains $\X$, then $\U=\M_\X$ or $\U=\mathcal N_{\X}$.
\end{thm}

\begin{proof}
By definition, $\X$ admits a two-term weak $\R[1]$-cluster tilting subcategory $\mathcal L=\add(\X\cup\{W\})$ where $W\notin \X$ is indecomposable. By Lemmas \ref{lem3} and \ref{lem4}, every object $R\in \R$ admits a triangle $R\xrightarrow{f} L^1_R\to L^2_R\xrightarrow{r} R[1]$ such that $L^1_R,L^2_R\in \mathcal L$ and $f$ is a minimal left $\mathcal L$-approximation. If for every $R$, $L^1_R,L^2_R\in \mathcal X$, then by Lemma \ref{lem4}, $\X$ itself becomes two-term weak $\R[1]$-cluster tilting, a contradiction. Hence $W$ must appear in some triangle $R\xrightarrow{f} L^1_R\to L^2_R\xrightarrow{r} R[1]$. Since $f$ is left minimal, by Lemma \ref{lem2}, $L^1_R$ and $L^2_R$ do not have common direct summands. Hence $W$ can not be a direct summand of both $L^1_R$ and $L^2_R$. If $W$ is a direct summand of $L^1_R$, then $W\in \mathcal N_\X$; if $W$ is a direct summand of $L^2_R$, then $W\in \M_\X$. Hence $\mathcal L=\M_\X$ or $\mathcal L=\mathcal N_{\X}$. Then by Lemma \ref{neq} and Lemma \ref{neq-D}, $\M_\X$ and $\N_\X$ are completions of $\X$.

Let $\U$ be a two-term weak $\R[1]$-cluster tilting subcategory which contains $\X$. Assume $\U\neq \M_\X$ and $\U\neq \mathcal N_{\X}$. Let $\mathcal M_\X=\add(\X\cup\{Y\})$ where $Y\notin \X$ is indecomposable. Then there is an object $R_0\in R$ which admits a triangle
$$R_0\to X_0\xrightarrow{x_0} M_0\xrightarrow{m_0} R_0[1]$$
where $X_0\in \X$, $M_0\in \M_\X$ and $Y$ is a direct summand of $M_0$. For any object $U\in \U$ and any morphism $\alpha: U[-1]\to R'\xrightarrow{r'} M_0$ with $R'\in \R$, we have $m_0r'=0$. Then there is a morphism $x':R'\to X_0$ such that $x_0x'=r'$. Hence $\alpha=0$, which implies that $[\R](\U[-1],\M_\X)=0$.

%There is an indecomposable object $R_0\in R$ which admits a triangle
%$$R_0\to U_1\to U_2\to R_0[1]$$
%where $U_1,U_2\in \U\backslash \X$. By Lemma \ref{alter} $R_0$ also admits a triangle
%$$R_0\to X_0\to M_0\to R_0[1]$$
%where $X_0\in \X$, $M_0\in \M_\X$ and $Y$ is a direct summand of $M_0$. Then we have the following commutative diagram
%$$\xymatrix{
%U_2[-1] \ar[r] \ar@{=}[d] &R_0 \ar[r] \ar[d] &U_1 \ar[r] \ar[d] &U_2 \ar@{=}[d]\\
%U_2[-1] \ar[r]^0 &X_0 \ar[r] \ar[d] &X_0\oplus U_2 \ar[r] \ar[d]^{\gamma} &U_2\\
%&M_0 \ar@{=}[r] \ar[d] &M_0 \ar[d]^{m_0}\\
%&R_0[1] \ar[r] &U_1[1]
%}
%$$
%Since $m_0$ factors through $R_0[1]$, for any object $U\in \U$ and any morphism $\alpha: U[-1]\to R'\xrightarrow{r'} M_0$, we have $m_0r'=0$. Then there is a morphism $\beta:R'\to X_0\oplus U_2$ such that $\gamma\beta=r'$. Hence $\alpha=0$, which implies that $[\R](\U[-1],\M_\X)=0$.

Now let $U_0$ be an indecomposable object in $\U\backslash \X$. It admits a triangle
$$U_0[-1]\to R_1\xrightarrow{r_1} R_2\to U_0$$
wit $R_1,R_2\in \R$. Since $[\R](\U[-1],\M_\X)=0$, $\Hom_\C(r_1,M)$ is surjective for any $M\in \M_\X$ by Lemma \ref{tau1}. $R_2$ admits a triangle
$$R_2\xrightarrow{x_2}X_2\to M_2'\to R_2[1]$$
where $x_2$ is a left $\X$-approximation and $M_2'\in \M_\X$. Then we have the following commutative diagram
$$\xymatrix{
&M_2'[-1] \ar@{=}[r] \ar[d] &M_2'[-1] \ar[d]\\
R_1 \ar[r]^r \ar@{=}[d] &R_2 \ar[r] \ar[d]^{x_2} &U_0 \ar[r] \ar[d]^{u'} &R[1] \ar@{=}[d]\\
R_1 \ar[r]^{r_1} &X_2 \ar[r] \ar[d] &M_1' \ar[r] \ar[d] &R[1]\\
&M_2' \ar@{=}[r] &M_2'
}
$$
where $r_1$ is a left $\X$-approximation. Then $M_1'\in \M_\X$. By Lemma \ref{lem3}, $U_0$ admits a triangle
$$U_0\xrightarrow{u_0} M_1 \xrightarrow{m_1} M_2\xrightarrow{m_2} U_0[1]$$
where $u_0$ is a left minimal $\M_\X$-approximation and $m_2$ factors through $\R[1]$. If $M_2\in \X$, then $m_2=0$, which implies that $U_0\in \M_\X$, a contradiction. Hence $Y$ is a direct summand of $M_2$. If $M_1\in \X$, then by Lemma \ref{lem3}, $m_1$ becomes a right $\X$-approximation of $M_2$. Then by Theorem \ref{m-pair}, $U_0\in \N_\X$, a contradiction. Hence $Y$ has to be a direct summand of $M_1$. But by Lemma \ref{lem2}, $M_1$ and $M_2$ have no common direct summands, a contradiction. Hence $\U$ has to be $\M_\X$ or $\N_\X$.
\end{proof}

%As a direct consequence of Theorem \ref{main1},  we obtain the following important result.
%
%\begin{cor}{\rm \cite[Theorem 3.5]{IY}}
%Let $\C$ be a $2$--Calabi-Yau triangulated category. Then
%any almost complete cluster tilting subcategory $\X$ of\hspace{1mm} $\C$ is
%contained in exactly two cluster tilting subcategories $\mathcal M_{\X}$ and $\mathcal N_{\X}$ of \hspace{1mm}$\C$.
%\end{cor}
%
%\proof In any $2$--Calabi-Yau triangulated category, by \cite[Lemma 3.5]{YZZ},
%we know that $\X$ is $\R[1]$-rigid if and only if $\X$ is rigid. Thus, this follows from Theorem \ref{main1}.  \qed

\section{$\tau$-tilting theory in functor categories}

From this section, we denote $\RR$ by $\A$. We assume that $\A/\R[1]$ is abelian. Note that when $\R$ is contravariantly finite, $\A/\R[1]$ is the heart of a cotorsion pair, hence is abelian, see \cite[Theorem 6.4]{Na}.
For convenience, denote $\A/\R[1]$ by $\overline {\mathcal A}$.  For any subcategory $\mathcal D\subseteq \A$, we denote its image through canonical quotient functor $\pi: \A\to \overline \A$ by $\overline {\mathcal D}$.

\vspace{2mm}

By \cite[Proposition 6.2]{IY}, the following functor
\begin{align*}
\mathbb{H}\colon &\C\longrightarrow \Mod \R\\
&M \longmapsto\Hom_\C(-,M)|_\R
\end{align*}
induces an equivalence $\overline \A\simeq \mod \R$.
\vspace{2mm}

%To prove our main results, we need the following some preparations.

\begin{lem}\label{r2}
Let $A\xrightarrow{f} B\xrightarrow{g} C\xrightarrow{h} A[1]$ be a triangle such that $A,B,C\in \mathcal A$ and ${\rm Hom}_{\C}(\R,h)=0$. Then we have an exact sequence $A\xrightarrow{\overline f} B\xrightarrow{\overline g} C\to 0$.
\end{lem}

\proof
Apply the functor ${\rm Hom}_{\C}(\R,-)$ to the triangle
 $$A\xrightarrow{f} B\xrightarrow{g} C\xrightarrow{h} A[1],$$
  we have the follow exact sequence:
  $${\rm Hom}_{\C}(\R,A)\xrightarrow{{\rm Hom}_{\C}(\R,f)}{\rm Hom}_{\C}(\R,B)
  \xrightarrow{{\rm Hom}_{\C}(\R,g)}{\rm Hom}_{\C}(\R,C)\xrightarrow{{\rm Hom}_{\C}(\R,~h)=0}{\rm Hom}_{\C}(\R,A[1]).$$
By the equivalence $\overline A\simeq \mod \R$, we can get the desired exact sequence.
             \qed

\begin{lem}\label{proj}
$\overline\R$ is the subcategory of enough projective objects in $\overline \A$.
\end{lem}

\begin{proof}
Let $\overline f:A\to B$ be an epimorphism in $\overline \A$. Then $\Hom_{\C}(R,A)\xrightarrow{ {\rm Hom}_{\C}(R,f)} {\rm Hom}_{\C}(R,B)$ is also an epimorphism for any object $R\in\R$. Hence we have the following commutative diagram
$$\xymatrix{
&R\ar[d] \ar@{.>}[ld]\\
A \ar[r]_{\overline f} &B\ar[r] &0,
}
$$
which implies that $\overline \R$ is a subcategory of projective objects in $\overline \A$.

On the other hand, any object $A\in \A$ admits a triangle $R_1\to R_0\to A\to R[1]$ with $R_0,R_1\in \R$, then by Lemma \ref{r2}, we have an epimorphism $R_0\to A$ in $\overline \A$. Hence $\overline \R$ is the subcategory of enough projective objects in $\overline \A$.
\end{proof}

Let $\widehat{\A}$ be an abelian category. For any subcategory $\X\subseteq \widehat{\A}$, denote by $\Fac \X$ the following subcategory of $\widehat{\A}$:
$$\{W\in  \widehat{\A}~\mid~ \text{there is an epimorphism }X\to W \text{ with }X\in  \X \}.$$

\begin{lem}\label{tau2}
Let $\U,\V$ be two subcategories of $\A$. Assume that $\U\cap \R[1]=0$. Then ${\rm Ext}^1_{\overline \A}(\overline \U,\Fac \overline \V)=0$ if and only if $[\R[1]](\U,\V[1])=0$.
\end{lem}

\begin{proof}
We show the ``if" part first.

By Lemma \ref{tau1}, any object $U\in U$ admits a triangle
$$R_1\xrightarrow{f} R_2\xrightarrow{u_1} U\xrightarrow{u_2} R_1[1]$$
where $\Hom_{\C}(f,V)$ is surjective for any object $V\in \V$. By Lemma \ref{r2}, we have an exact sequence in $\overline \A$:
$$R_1\xrightarrow{\overline f} R_2\xrightarrow{\overline u_1} U\to 0.$$
Let $R_1\xrightarrow{\overline f_1} Y\xrightarrow{\overline f_2} R_2$ be an epic-monic factorization of $\overline f$. Let $\overline y:Y\to W$ be any morphism with $W\in \Fac \overline \V$. $W$ admits an epimorphism $\overline w:V\to W$ with $V\in \V$. Since $R_1$ is a projective object in $\overline \A$ by Lemma \ref{proj}, there is a morphism $\overline r_1:R_1\to V$ such that $\overline w\overline r_1=\overline y\overline f_1$. Since $\Hom_{\overline \A}(\overline f,V)$ is surjective, there is a morphism $\overline r_2:R_2\to V$ such that $\overline r_1=\overline r_2\overline f$. Then $\overline y\overline f_1=\overline w\overline r_2\overline f$. Since $\overline f_1$ is an epimorphism, we have $\overline y=\overline w\overline r_2\overline f_2$. Then we have the following exact sequence
$$\Hom_{\overline \A}(R_2,W)\to {\rm Hom}_{\overline \A}(Y,W)\xrightarrow{0} {\rm Ext}^1_{\overline \A}(U,W)\to {\rm Ext}^1_{\overline \A}(R_2,W)=0.$$
Hence ${\rm Ext}^1_{\overline \A}(U,W)=0$, which implies ${\rm Ext}^1_{\overline \A}(\overline \U,\Fac \overline \V)=0$.
\vspace{1mm}

Now we show the ``only if" part.
\vspace{1mm}

Let $U\in \U$ be an indecomposable object. It admits a triangle
$$R_1\xrightarrow{f} R_2\xrightarrow{u_1} U\xrightarrow{u_2} R_1[1].$$
By Lemma \ref{tau1}, we need to check that $\Hom_{\C}(f,V)$ is surjective for any object $V\in \V$. Let $r_1:R_1\to V$ be any morphism. Then we have the following commutative diagram.
$$\xymatrix{
U[-1] \ar[r] \ar@{=}[d] &R_1 \ar[r]^f \ar[d]^{r_1} &R_2 \ar[r]^{u_1} \ar[d] &U \ar@{=}[d]\\
U[-1] \ar[r]^u &V \ar[d] \ar[r]^v &A_2 \ar[r]^a \ar[d] &U\\
&A_1 \ar@{=}[r] \ar[d] &A_1\ar[d]\\
&R_1[1] \ar[r] &R_2[1]
}
$$
Note that $\overline a\neq 0$, otherwise $u_1=0$, which implies that $U_1\in \R[1]$, a contradiction. By Lemma \ref{r1}, $A_1\in \mathcal A$. Again by Lemma \ref{r1}, $A_2\in \mathcal A$. Since $u[1]$ factors through $R_1[1]$, we have an exact sequence
$$V\xrightarrow{\overline v} A_2\xrightarrow{\overline a} U\to 0$$
Let $V\xrightarrow{\overline v_1} W\xrightarrow{\overline v_2}A_2$ be an epic-monic factorization of $\overline v$. Since $W\in \Fac \overline \V$ and ${\rm Ext}^1_{\overline \A}(\overline \U,\Fac \overline \V)=0$, $\overline a$ becomes a retraction. But $U$ is indecomposable, which means $a$ is also a retraction. Then $u=0$ and $r_1$ factors through $f$.
\end{proof}

We can get the following corollary immediately.

\begin{cor}\label{ex}
Let $X,Y\in \A$. Then ${\rm Hom}_{\C}(X,Y[1])=0$ implies that ${\rm Ext}^1_{\overline\A}(X,Y)=0$
\end{cor}

We have the following useful proposition.

\begin{prop}\label{tau3}
Let $\U$ be a subcategory of $\A$ which is closed under direct sums and direct summands. Assume that $\U\cap \R[1]=0$. Then the following are equivalent:
\begin{itemize}
\item[\rm (1)] $\U$ is two-term $\R[1]$-rigid.
\item[\rm (2)] ${\rm Ext}^1_{\overline \A}(\overline \U,\Fac \overline \U)=0$.
\item[\rm (3)] Any object $U\in \U$ admits a triangle
$$R_1\xrightarrow{f} R_2\to U\to R_1[1]$$
where $\Hom_{\C}(f,U')$ is surjective for any object $U'\in \U$.
\item[\rm (4)] Any object $U\in \U$ admits an exact sequence
$$R_1\xrightarrow{\overline f} R_2\to U\to 0$$
where ${\rm Hom}_{\overline \A}(\overline f,U')$ is surjective for any object $U'\in \U$.
\end{itemize}
\end{prop}

\begin{proof}
By Lemma \ref{tau1} and Lemma \ref{tau2}, (1), (2) and (3) are equivalent to each other. By Lemma \ref{r2}, (3) implies $(4)$. We only need to show (4) implies (3).

Since for any $R\in \R$, ${\rm Hom}_\C(R,U')={\rm Hom}_{\overline \A}(R,U')$, ${\rm Hom}_{\overline \A}(\overline f,U')$ is surjective implies that ${\rm Hom}_{\C}(f,U')$ is surjective. We have a triangle $R_1\xrightarrow{f} R_2\to U_0\to R_1[1]$. It induces an exact sequence $R_1\xrightarrow{\overline f} R_2\to U_0\to 0$ in $\overline \A$. Hence $U\cong U_0$ in $\overline \A$. Since $U$ has no direct summand in $\R[1]$, $U_0=U\oplus R'[1]$ with $R'\in \R$. Since ${\rm Hom}_{\C}(R_2,R'[1])=0$, we can find that $R_1\xrightarrow{f} R_2\to U_0\to R_1[1]$ is the direct sum of triangle $R_1'\xrightarrow{f'} R_2\to U\to R_1'[1]$ and $R'\to 0\to R'[1]\xrightarrow{1} R'[1]$. Hence $U$ admits a triangle $R_1'\xrightarrow{f'} R_2\to U\to R_1'[1]$ where $R_1'\in \R$ and $\Hom_{\C}(f',U')$ is surjective for any object $U'\in \U$.
\end{proof}

\begin{defn}\label{taupair}
Let $\widehat{\A}$ be an abelian category with enough projectives. A subcategory $\U$ is called a $\tau$-rigid subcategory if ${\rm Ext}^1_{\widehat{\A}}(\U,\Fac\U)=0$; an object $U$ is called a $\tau$-rigid object if $\add U$ is a $\tau$-rigid subcategory. A pair $(\U,\mathcal Q)$ is called a $\tau$-rigid pair if
\begin{itemize}
\item[\rm (1)] $\U$ is $\tau$-rigid;
\item[\rm (2)] $\mathcal Q\subseteq \{ P \text{ is projective}~|~ {\rm Hom}_{\widehat{\A}}(P,\U)=0\}$.
\end{itemize}
A subcategory $\M$ is called a support $\tau$-tilting subcategory if
\begin{itemize}
\item[\rm (a)] $\M$ is $\tau$-rigid;
\item[\rm (b)] every projective object $P$ admits an exact sequence
$$P\xrightarrow{p} M_1\to M_2\to 0$$
where $M_1,M_2\in \M$ and $p$ is a left $\M$-approximation. An object $M$ is called a support $\tau$-tilting object if $\add M$ is a support $\tau$-tilting subcategory.
\end{itemize}
A pair $(\M,\E)$ is called a support $\tau$-tilting pair if it is a $\tau$-rigid pair, $\M$ is support $\tau$-tilting and $\E=\{ P \text{ is projective}~|~ {\rm Hom}_{\widehat{\A}}(P,\M)=0\}$. We say a support $\tau$-tilting pair $(\M,\E)$ contains a $\tau$-rigid pair $(\U,\mathcal Q)$ if $\U\subseteq \M$ and $\mathcal Q\subseteq \E$.

A $\tau$-rigid pair $(\U,\mathcal Q)$ is called an almost complete support $\tau$-tilting pair if there exists an indecomposable object $X$ such that $(\add(\U\cup \{X\}),\mathcal Q)$ or $(\U,\add(\mathcal Q\cup\{X\})$ is a support $\tau$-tilting pair.
\end{defn}

\begin{rem}
By definition, if $(\U,\mathcal Q)$ is a $\tau$-rigid pair in $\widehat{\A}$, so is $(\add\U,\add \mathcal Q)$. For convenience, we always assume that $(\U,\mathcal Q)=(\add\U,\add \mathcal Q)$ for a $\tau$-rigid pair.
\end{rem}

For the convenience of the readers, we give a proof for the following theorem (\cite[Theorem 4.4 and Theorem 4.5]{ZhZ}).

\begin{thm}\label{thm1}
Let $\X$ be a subcategory of $\A$. Then the following statements hold.
\begin{itemize}
\item[\rm (1)] $(\overline \X,\overline{\E})$ is a $\tau$-rigid pair if and only if $\X$ is two-term $\R[1]$-rigid and $\E\subseteq \R(\X)$;
  \vspace{1mm}

\item[\rm (2)] If $\X$ is two-term weak $\R[1]$-cluster tilting, then $(\overline \X,\overline {\R(\X)})$ is a support $\tau$-tilting pair;
    \vspace{1mm}

\item[\rm (3)] If $(\overline \X,\overline {\R(\X)})$ is a support $\tau$-tilting pair and $\R(\X)[1]\subseteq \X$, then $\X$ is two-term weak $\R[1]$-cluster tilting.
\end{itemize}
\end{thm}

\begin{proof}
Let $\X'$ be the subcategory of $\X$ consisting of the objects which do not have direct summands in $\R[1]$. Note that $\overline \X'=\overline \X$ and $\R(\X)=\R(\X')$.
\vspace{1mm}

(1) If $\X$ is two-term $\R[1]$-rigid, so is $\X'$. By Proposition \ref{tau3}, $\overline \X$ is $\tau$-rigid. Then by definition $(\overline \X,\overline \E)$ is a $\tau$-rigid pair when $\E\subseteq \R(\X)$. If $(\overline \X,\overline \E)$ is a $\tau$-rigid pair, then $\E\subseteq \R(\X)$. By Proposition \ref{tau3}, $\X'$ is two-term $\R[1]$-rigid. Then by Lemma \ref{lem0}, $\X$ is also two-term $\R[1]$-rigid.
\vspace{2mm}

(2) If $\X$ is two-term weak $\R[1]$-cluster tilting, by (1), Lemma \ref{lem3}, Lemma \ref{r2} and Lemma \ref{cor0}, $(\overline \X,\overline {\R(\X)})$ is a support $\tau$-tilting pair.
\vspace{2mm}

(3) If $(\overline \X,\overline \E)$ is a support $\tau$-tilting pair, Then $\X$ is two-term $\R[1]$-rigid by (1). Since each $R\in \R$ admits an exact sequence $R\xrightarrow{\overline x}X_1\to X_2\to 0$ where $X_1,X_2\in \X$ and $\overline x$ is a left $\overline \X$-approximation, we have a triangle $R\xrightarrow{x} X_1\to X_2'\to R[1]$. We can assume that $X_1,X_2\in \X'$. By Lemma \ref{r2}, $X_2\cong X_2'$ in $\overline \A$. Hence $X_2'\cong X_2\oplus R'[1]$ with $R'\in \R$. Note that $x$ is in fact a left $\X$-approximation, we have $X_2'\in \M_\X$. Hence $R'\in \R(\X)$ by Lemma \ref{cor0}. When $\R(\X)[1]\subseteq \X$, by Lemma \ref{lem4} $\X$ is two-term weak $\R[1]$-cluster tilting.
\end{proof}

\begin{lem}\label{7-0}
If $\overline\U$ is a $\tau$-rigid subcategory in $\overline \A$, then $\Fac\overline \U$ is closed under extensions.
\end{lem}

\begin{proof}
Let $0\to V_1\xrightarrow{\overline v_1} V\xrightarrow{\overline v_2} V_2\to 0$ be a short exact sequence in $\overline \A$ with $V_1,V_2\in \Fac\overline\U$. Assume that $V_i$ ($i=1,2$) admits an epimorphism $U_i\xrightarrow{\overline u_i} V_i\to 0$. By Lemma \ref{tau2}, $\Ext^1_{\overline \A}(U_2,V_1)=0$, then there is a morphism $\overline v:U_2\to V$ such that $\overline v_2\overline v=\overline u_2$. Hence we have the following commutative diagram of short exact sequences:
$$\xymatrix{
0 \ar[r] &U_1 \ar[r]^-{\svecv{1}{0}} \ar[d]^{\overline u_1} &U_1\oplus U_2 \ar[r]^-{\svech{0}{1}} \ar[d]^-{\svech{\overline v_1}{\overline v}} &U_2 \ar[r] \ar[d]^{\overline u_2} &0\\
0 \ar[r] &V_1 \ar[r]^{\overline v_1} &V \ar[r]^{\overline v_2} &V_2 \ar[r] &0.
}
$$
Since $\overline u_i$ ($i=1,2$) is an epimorphism, $\svech{\overline v_1}{\overline v}$ is also an epimorphism. Hence $V\in \Fac\overline\U$.
\end{proof}

We introduce the following notions. Let $\widehat{\A}$ be an abelian category with enough projectives. For any subcategory $\T\subseteq \widehat{\A}$, denote
$$\{A\in \widehat{\A} ~|~ {\rm Ext}^1_{\widehat{\A}}(\U, \Fac (\add A))=0\} \text{ by } {^{\bot}(\tau\U)};$$
denote
$$\{A\in \widehat{\A} ~|~ {\rm Hom}_{\widehat{\A}}(\U,A)=0\} \text{ by } \U^{\bot}.$$

\begin{lem}\label{6-1}
Let $\X$ be a two-term $\R[1]$-rigid subcategory, $\E\subseteq \R(\X)$. Then ${^{\bot}}(\tau \overline \X)\cap \overline \E^{\bot}\supseteq \Fac \overline \X$.
\end{lem}

\begin{proof}
If $\X\subseteq \R[1]$, $\overline \X=0$, hence ${^{\bot}}(\tau \overline \X)\cap \overline \E^{\bot}\supseteq \Fac \overline \X$. Now we assume that $X\nsubseteq \R[1]$.

Since $\overline \R$ is the subcategory of projective objects in $\overline \A$, any morphism from $\overline \R$ to $\Fac \overline \X$ factors through $\overline \X$. Hence $\overline \E^{\bot}\supseteq \Fac \overline \X$. Let $W\in \Fac \overline \X$ be any indecomposable non-zero object. By Lemma \ref{tau2}, to get that $W\in {^{\bot}}(\tau \overline \X)$, we only need to check that $[\R](X[-1],W)=0$ for any indecomposable object $X\in \X \backslash \R[1]$.

Since $W\in \Fac \overline \X$, then $W$ admits a triangle
$$U\to X'\xrightarrow{x'} W\xrightarrow{w} U[1]$$
where $X'\in\X$ and ${\rm \Hom}_{\C}(\R,w)=0$. Let $\alpha: X[-1]\xrightarrow{a_1} R \xrightarrow{a_2} W$ be any morphism in $[\R](X[-1],W)$. Since $wa_2=0$, there is a morphism $r:R\to X'$ such that $x'r=a_2$. Since $ra_1\in [\R](X[-1],X')=0$, we have $\alpha=0$. Thus ${^{\bot}}(\tau \overline \X)\supseteq \Fac \overline \X$.
\end{proof}

%
%We have the following corollary.
%
%\begin{cor}
%Let $\X$ be a two-term $\R[1]$-rigid subcategory such that $\X\nsubseteq \R[1]$. Assume $\X=\X'\cup \R'[1]$ where $\R'\subseteq \R$ and any object in $\X'$ has no direct summand in $\R[1]$. We also assume that $\X'$ is $\R[1]$-functorially finite. For convenience, we still use $\mathbf{P}({^{\bot}}(\tau \overline \X))$ to denote the subcategory in $\A$ which has the same object as the one in $\overline \A$. Assume $\mathbf{P}({^{\bot}}(\tau \overline \X))\cap \R[1]=0$. Then in $\A$, $\mathbf{P}({^{\bot}}(\tau \overline \X))$ is two-term weak $\R[1]$-cluster tilting. Moreover, ${\rm Hom}_{\C}(R,\mathbf{P}({^{\bot}}(\tau \overline \X)))\neq 0$ for any $R\in \R$.
%\end{cor}
%
%\begin{proof}
%Since $\overline \X=\overline \X'$, by Theorem \ref{thm2}, $\mathcal N_{\X'}=\mathbf{P}({^{\bot}}(\tau \overline \X))$. Hence $\mathbf{P}({^{\bot}}(\tau \overline \X))$ is two-term weak $\R[1]$-cluster tilting. By Lemma \ref{cap}, $\mathbf{P}({^{\bot}}(\tau \overline \X))\cap \R[1]=\mathcal N_{\X'}\cap \R[1]=\X'\cap \R[1]=0$. Then by Lemma \ref{cor0}, ${\rm Hom}_{\C}(R,\mathbf{P}({^{\bot}}(\tau \overline \X)))\neq 0$ for any $R\in \R$.
%\end{proof}

In summary, we have the following main result.

\begin{thm}\label{main}
Let $\X$ be an $\R[1]$-functorially finite two-term $\R[1]$-rigid subcategory. Denote $\X[-1]\cap\R$ by $\E$. If $(\overline{\X}, \overline\E)$ is not a support $\tau$-tilting pair in $\overline{\A}$, then it is contained in two support $\tau$-tilting
pairs $(\overline \M_\X, \overline{\R(\X)})$ and $(\overline \N_\X, \overline\E)$ such that
$$\Fac\overline\M_\X=\Fac \overline \X,~ \Fac\overline\N_\X={^\perp(\tau\overline{\X})}\cap{\overline\E^{\perp}}.$$
Moreover, if $\X$ is almost complete, then $(\overline \M_\X, \overline{\R(\X)})$ and $(\overline \N_\X, \overline\E)$ are the only support $\tau$-tilting pairs that contain $(\overline{\X}, \overline\E)$.
\end{thm}

\proof By Proposition \ref{facmax} and Proposition \ref{NX}, $\X$ is contained in $2$ two-term weak $\R[1]$-cluster tilting subcategories $\M_\X$ and $\mathcal N_\X$.

By Lemma \ref{cap}(1) and Lemma \ref{cor0}, we have
$$\M_\X\cap\R[1]=\R(\X)[1]\supseteq \E[1].$$
%Hence $(\overline \M_{\X}, \overline\E)$ and $(\overline \N_{\X}, \overline\E)$ are support $\tau$-tilting pairs that contain $(\overline \X, \overline\E)$ by Theorem \ref{thm1}.
%
%If $\E\subsetneq \R(\X)$, by Lemma \ref{cap}(1), we have
%$$\R(\X)[1]=\M_\X\cap \R[1]\supsetneq \E[1].$$
Then $(\overline \M_\X, \overline {\R(\X)})$ is a support $\tau$-tilting pair that contain $(\overline {\X}, \overline\E)$. On the other hand, by Lemma \ref{cap}(2), $\mathcal N_\X\cap\R[1]=\E[1]$. Hence $(\overline{\N}_{\X}, \overline\E)$ is also a support $\tau$-tilting pair that contains $(\overline{\X}, \overline\E)$ by Theorem \ref{thm1}.

Since $\overline \X\subseteq \overline \M_\X\subseteq \Fac \overline \X$, we have $\Fac \overline \M_\X=\Fac \overline \X$.

Since $\mathcal N_\X\cap\R[1]=\X\cap \R[1]$, by Lemma \ref{6-1}, $\Fac \overline {\mathcal N}_\X\subseteq {^{\bot}}(\tau \overline {\mathcal N}_\X)\cap \overline \E^{\bot}\subseteq {^{\bot}}(\tau \overline \X)\cap \overline \E^{\bot}$. Let $L\in {^{\bot}}(\tau \overline \X)\cap \overline \E^{\bot}$ be an indecomposable object.
$L$ admits a triangle $R_1\to R_0\to L\to R_1[1]$ with $R_1,R_0\in \R$. $R_0$ admits a triangle $R_0\to N_0\to X_0\to R_0[1]$ with $N_0\in \mathcal N_\X$ and $X_0\in \X$. Then we have the following commutative diagram.
$$\xymatrix{
R_1 \ar[r] \ar@{=}[d] &R_0 \ar[r] \ar[d] &L \ar[r] \ar[d] &R_1[1] \ar@{=}[d] \\
R_1 \ar[r] &N_0 \ar[r] \ar[d] &A \ar[r] \ar[d] &R_1[1]\\
&X_0 \ar@{=}[r] \ar[d] &X_0 \ar[d]^{\alpha}\\
&R_0[1] \ar[r] &L[1]
}
$$
Assume $X_0=X_0'\oplus R_0'[1]$, where $R_0'\in \R$ and $X_0'$ has no direct summand in $\R[1]$. Since $R_0'\in\E$, we have ${\rm Hom}_{\C}(R_0',L)=0$. By Lemma \ref{tau2}, $L\in {^{\bot}}(\tau \overline \X)$ implies that $[\R[1]](X_0',L[1])=0$. Hence $\alpha=0$, which implies $A=L\oplus X_0$. Since we have an epimorphism $N_0\to A$ in $\overline \A$, $L\in \Fac\overline {\mathcal N}_\X$.
%Any indecomposable object $N_0\in \mathcal N_\X\backslash \X$ admits a triangle
%$$X_1[-1]\to R_1\xrightarrow{r_1} N_1\xrightarrow{x_1} X_1$$
%where $N_1\in \mathcal N_{\X}$, $\X_1\in \X$ and $N_0$ is a direct summand of $N_1$. For any morphism $\alpha: N_1 \xrightarrow{n_1} R[1] \xrightarrow{l_1} L[1]$, since $n_1r_1=0$, there is a morphism $x:X_1\to R[1]$ such that $n_1=xx_1$. Assume $X_1=X_1'\oplus R_1'[1]$, where $R_1'\in \R$ and $X_1'$ has no direct summand in $\R[1]$. Since $R'\in\E$, we have ${\rm Hom}_{\C}(R_1',L)=0$. By Lemma \ref{tau2}, $L\in {^{\bot}}(\tau \overline \X)$ implies that $[\R[1]](X_1',L[1])=0$. Hence $l_1x_1=0$, which implies $\alpha=0$. Thus $L\in {^{\bot}}(\tau \overline {\mathcal N}_\X)$ and $\Fac \overline {\mathcal N}_\X={^{\bot}}(\tau \overline \X)\cap \overline \E^{\bot}$.

The ``Moreover" part is followed by Theorem \ref{main1} and Theorem \ref{thm1}.
\qed

\section{$\tau$-tilting theory in abelian categories}

\begin{defn}
A pair of subcategories $(\T,{\mathcal F})$ in an abelian category $\widehat{\A}$ is called a torsion pair if the following conditions are satisfied:
\begin{itemize}
\item[\rm (1)] ${\rm Hom}_{\widehat{\A}}(\T,\mathcal F)=0$;
\item[\rm (2)] any object $A\in \widehat{\A}$ admits a short exact sequence $0\to T\to A\to F\to 0$ where $T\in\T$ and $F\in {\mathcal F}$.
\end{itemize}
A subcategory $\T\subseteq  \widehat{\A}$ is called \emph{torsion class} if it admits a torsion pair $(\T,{\mathcal F})$.
\end{defn}

\begin{lem}\label{conf}
Let $\overline \X$ be a contravariantly finite $\tau$-rigid subcategory in $\overline \A$. Then $(\Fac \overline \X,\overline \X^{\bot})$ is a torsion pair. Moreover, if any object $R\in \R$ admits a left $\overline \X$-approximation, then $\Fac \overline \X$ is functorially finite.
\end{lem}

\begin{proof}
For convenience, we can assume that $\X\cap\R[1]=0$. It is obvious that ${\rm Hom}_{\overline \A}(\Fac \overline \X,\overline \X^{\bot})=0$. Let $A\in \overline \A$ be any object. Assume that $A$ has no direct summand in $\R[1]$.

Since $\overline \X$ is contravariantly finite in $\overline \A$, $A$ admits a right $\overline \X$-approximation $\overline x:X\to A$. Then we have a triangle
$$B[-1]\to X\xrightarrow{x} A\to B.$$
$X$ admits a triangle
$$X\xrightarrow{r_1} R_1[1]\xrightarrow{r_2} R_2[1]\xrightarrow{x_1} X[1]$$
where $R_1,R_2\in \R$. Then we have the following commutative diagram.
$$\xymatrix{
B[-1] \ar@{=}[d] \ar[r] &X \ar[r]^x \ar[d]^{r_1} &A \ar[r] \ar[d]^a &B\ar@{=}[d] \\
B[-1] \ar[r] &R_1[1] \ar[r]^{u_1} \ar[d]^{r_2} &U \ar[r] \ar[d]^{u_2} &B\\
&R_2[1] \ar@{=}[r] \ar[d]^{x_1} &R_2[1]\ar[d]\\
&X[1] \ar[r]_{x[1]} &A[1]
}
$$
By Lemma \ref{r1}, $U\in \A$. We can get a triangle $X \xrightarrow{\svecv{x}{r_1}} A\oplus R_1[1]\xrightarrow{\svech{-a}{u_1'}} U \xrightarrow{x_1u_2} X[1]$. Then we have an exact sequence $X\xrightarrow{\overline x} A\xrightarrow{-\overline a} U\to 0$ in $\overline \A$. Let $X\xrightarrow{\overline c_1}C\xrightarrow{\overline c_2}A$ be the epic-monic factorization of $\overline x$. Then $C\in \Fac \overline \X$ and we have a short exact sequence
$$0\to C\xrightarrow{\overline c_2} A\to U\to 0.$$
We show $U\in \overline \X^{\bot}$.

Let $X_0$ be any object in $\X$ and $f\in {\rm Hom}_{\C}(X_0,U)$. Since $\X$ is $\R[1]$-rigid, $x_1u_2f=0$. Then there is a morphism $x_0:X_0\to R_1[1]$ such that $r_2x_0=u_2f$. Hence $0=r_2x_0-u_2f_2=u_2(u_1x_0-f_2)$ and there exists a morphism $x_0':X_0\to A$ such that $ax_0'=u_1x_0-f_2$. Since $\overline x$ is  a right $\overline \X$-approximation, there is a morphism $\overline x':X_0\to X$ such that $\overline x_0'=\overline x\overline x'$. Hence
$$\overline f_2=\overline u_1\overline x_0-\overline a\overline x\overline x'=\overline u_1\overline x_0-\overline u_1\overline x'=0,$$
which implies $U\in \overline \X^{\bot}$.

Now we assume that any object $R\in \R$ admits a left $\overline \X$-approximation. Note that this means any object $R\in \R$ admits a left $\X$-approximation. Since $\Fac \overline \X$ is already contravariantly finite, we show that it is covariantly finite in $\overline \A$.

$A$ admits a triangle $R_0\to R_A\xrightarrow{r_A} A\to R_0[1]$ with $R_0,R_A\in \R$. $R_A$ admits a triangle $M[-1]\to R_A\xrightarrow{r'} X_1\to M_1$ where $r'$ is a left $\X$-approximation. Then we have the following commutative diagram.
$$\xymatrix{
&M_1[-1] \ar@{=}[r] \ar[d] &M_1[-1] \ar[d]\\
R_0 \ar[r] \ar@{=}[d] &R_A \ar[r]^{r_A} \ar[d]^{r'} &A \ar[r] \ar[d]^{a_1} &R_0[1] \ar@{=}[d]\\
R_0 \ar[r] &X_1 \ar[r] \ar[d] &C_1 \ar[r] \ar[d] &R_0[1]\\
&M_1 \ar@{=}[r] &M_1
}
$$
Note that $C_1\in \Fac \overline \X$, we show that $\overline a_1$ is a left $(\Fac\overline \X)$-approximation.

Let $C'$ be any object in $\Fac \overline \X$. Then $C'$ admits a triangle $D[-1]\to X'\xrightarrow{x'} C'\xrightarrow{c'} D$ where $X'\in \X$ and ${\rm Hom}_{\C}(\R,c')=0$. Let $a'\in {\rm Hom}_{\C}(A,C')$. Since $c'a'r_A=0$, there is a morphism $x_A:R_A\to X'$ such that $x'x_A=a'r_A$. Since $r'$ is a left $\X$-approximation, there is a morphism $x_1':X_1\to X'$ such that $x_1'r'=x_A$. Hence $x'x_1'r'=a'r_A$ and there exists a morphism $c_1':C_1\to C'$ such that $c_1'a_1=a'$. Hence $\overline a_1$ is a left $(\Fac\overline \X)$-approximation.
\end{proof}

\begin{defn}\label{lcp}
Let $\widehat{\A}$ be an abelian category with enough projectives. Let $(\U,\V)$ be a pair of subcategories in $\widehat{\A}$. Consider the following conditions:
\begin{itemize}
\item[\rm (a1)] $\U=\{A\in \widehat{\A}~|~{\rm Ext}^1_{\widehat{\A}}(A,\V)=0\}:={^{\bot_1}\V}$.
\item[\rm (a2)] Any projective object $P$ admits an exact sequence $P\xrightarrow{p} V^P\to U^P\to 0$ where $U^P\in \U$, $V^P\in \U\cap\V$ and $p$ is a left $\V$-approximation.
\item[\rm (b1)] ${\rm Ext}^1_{\widehat{\A}}(\U,\V)=0.$
\item[\rm (b2)] Any object $A\in\widehat{\A}$ admits two exact sequences
$$0\to V_A\to U_A\to A\to 0~~~\mbox{and}~~~A\xrightarrow{v} V^A\to U^A\to 0$$
where $U_A,U^A\in \U$, $V_A,V^A\in \V$ and $v$ is a left $\V$-approximation.
\item[\rm (c1)] $\Fac\V=\V$.
\item[\rm (c1$'$)] $\V$ is a torsion class.
\item[(c2)] $\V$ is closed under extensions.
\end{itemize}
$(\U,\V)$ is called a $\tau$-cotorsion pair ({\rm \cite[Definition 2.9]{PZZ}}) if conditions {\rm (a1)} and {\rm (a2)} are satisfied; it is called a $\tau$-cotorsion torsion pair ({\rm \cite[Definition 4.4]{AST}}) if in addition, {\rm (c1$'$)} is satisfied. $(\U,\V)$ is called a left weak cotorsion pair ({\rm \cite[Definition 0.2]{BZ}, \cite[Remark 4.9]{AST}}) if conditions {\rm (b1)} and {\rm (b2)} are satisfied; it is called a left weak cotorsion torsion pair ({\rm \cite[Remark 4.9]{AST}}) if in addition, {\rm (c1$'$)} is satisfied.
\end{defn}

We recall the following result.

\begin{thm}{\rm (\cite[Theorem 3.8]{PZZ}, \cite[Theorem 5.7]{AST})}\label{PZZ}
Let $\widehat{\A}$ be a Krull-Schmidt abelian category with enough projectives. There are mutually inverse bijections:
\begin{align*}
\{ \text{ support }\tau\text{-tilting subcategories } \}&~\leftrightarrow~ \{ ~\tau\text{-cotorsion pairs }(\U,\V) ~|~ \V \text{ satisfies } {\rm (c1)},{\rm (c2)} \}\\
\M &~\mapsto~ ({^{\bot_1}}(\Fac \M), \Fac \M)  \\
\U\cap\V&\hspace{2mm}\raisebox{1.2ex}{\rotatebox{180}{$\mapsto$}}\hspace{2mm} (\U,\V)
\end{align*}
Moreover, these bijections induce the following one-to-one correspondence:
$$\{ \text{ contravariantly finite support }\tau\text{-tilting subcategories } \}\leftrightarrow \{ ~\tau\text{-cotorsion torsion pairs }  \}$$
\end{thm}

By Theorem \ref{PZZ}, we can get the following corollary.

\begin{cor}\label{tau4}
Let $\X$ be an $\R[1]$-functorially finite two-term $\R[1]$-rigid subcategory such that $\X[-1]\cap\R= \R(\X)$. Then $\X$ is two-term weak $\R[1]$-cluster tilting if and only if ${^{\bot}}(\tau \overline \X)\cap \overline {\R(\X)}^{\bot}=\Fac \overline \X$.
\end{cor}

\begin{proof}
If $\X$ is two-term weak $\R[1]$-cluster tilting, by Theorem \ref{main}, $\Fac \overline \X=\Fac\overline\N_\X={^{\bot}}(\tau \overline \X)\cap \overline {\R(\X)}^{\bot}$. If ${^{\bot}}(\tau \overline \X)\cap \overline {\R(\X)}^{\bot}=\Fac \overline \X$, then by Theorem \ref{main}, $\Fac \overline \M_\X=\Fac\overline\N_\X$. Thus by Theorem \ref{PZZ}, we have $\overline \M_\X=\overline \N_\X$. By Corollary \ref{capcap}, $\X=\M_\X=\N_\X$, hence $\X$ is two-term weak $\R[1]$-cluster tilting.
\end{proof}

According to Theorem \ref{PZZ}, we can define a partial order of support $\tau$-tilting subcategories in $\widehat{\A}$:

\begin{defn}
Let $\M$ and $\N$ be two support $\tau$-tilting subcategories in $\widehat{\A}$. We write $\M\geq\N$ when $\Fac \M\supseteq \Fac \N$.
\end{defn}

\begin{prop}\label{partial}
Let $\X$ be an $\R[1]$-functorially finite two-term $\R[1]$-rigid subcategory. Denote $\X[-1]\cap \R$ by $\E$. Let $\mathcal L$ be a two-term weak $\R[1]$-cluster tilting subcategory. Then $\mathcal L\supseteq \X$ if and only if $\overline \N_\X\geq \overline {\mathcal L}\geq \overline \M_\X$.
\end{prop}

\begin{proof}
We assume $\X\subseteq \mathcal L$ first. Since $\Fac \overline \M_\X=\Fac \overline \X$ by Theorem \ref{main}, we have $\Fac \overline {\mathcal L}\supseteq \Fac \overline \M_\X$. Hence $\overline {\mathcal L}\geq \overline \M_\X$. Since $\E\subseteq \X\subseteq \mathcal L$, by Lemma \ref{cor0} we have $\mathcal L\subseteq \E^{\bot}$. On the other hand, $\X\subseteq \mathcal L$ implies that $[\R[1]](\X,\mathcal L[1])$, then by Lemma \ref{tau2}, $\Ext^1_{\overline \A}(\overline \X, \Fac \overline {\mathcal L})=0$. Then $\overline {\mathcal L}\subseteq {^{\bot}}(\tau\overline \X)\cap \overline \E^{\bot}$. Hence by Theorem \ref{main}, $\overline \N_\X\geq \overline {\mathcal L}$.

Now we assume $\overline \N_\X\geq \overline {\mathcal L}\geq \overline \M_\X$. Since $\mathcal L\subseteq \E^{\bot}$, we have $\E[1]\subseteq \mathcal L$. Since $\overline {\mathcal L}\geq \overline \M_\X$, we have $\overline \X\subseteq \Fac\overline {\mathcal L}$. Then any morphism form  $\R$ to $\X$ factors through $\mathcal L$. Hence $\R(\X)[1]\supseteq \R(\mathcal L)[1]=\mathcal L\cap \R[1]$. We also have $\Ext^1_{\overline \A}(\overline {\mathcal L},\Fac\overline \X)=0$, which implies that $[\R[1]](\mathcal L,X[1])=0$ for any indecomposable object $X\in \X\backslash \R[1]$. On the other hand, $\overline \N_\X\geq \overline {\mathcal L}$ implies that $\Ext^1_{\overline \A}(\overline \X, \Fac \overline {\mathcal L})=0$. Hence $[\R[1]](X,\mathcal L[1])=0$ for any indecomposable object $X\in \X\backslash \R[1]$. Since $\mathcal L$ is a two-term weak $\R[1]$-cluster tilting subcategory, we have $\mathcal L\supseteq \X$.
\end{proof}

We also have the following corollary

\begin{cor}
Let $\X$ be an $\R[1]$-functorially finite two-term $\R[1]$-rigid subcategory. If $\X$ is not two-term weak $\R[1]$-cluster tilting, then $\Fac\overline\M_\X\subsetneq \Fac\overline\N_\X$.
\end{cor}

\begin{proof}
By Proposition \ref{partial}, $\Fac\overline\M_\X\subseteq \Fac\overline\N_\X$. If $\Fac\overline\N_\X=\Fac\overline\M_\X$, then by Theorem \ref{main}, ${^{\bot}}(\tau \overline \X)\cap \overline {\R(\X)}^{\bot}=\Fac \overline \X$. By Corollary \ref{tau4}, $\X$ itself becomes two-term weak $\R[1]$-cluster tilting. Hence if $\X$ is not two-term weak $\R[1]$-cluster tilting, we can get $\Fac\overline \M_\X\subsetneq \Fac\overline \N_\X$.
\end{proof}

\vspace{2mm}

{\bf\hspace{-6mm} 6.1~~When $\R$ is $2$-rigid.}\hspace{3mm}
We call a subcategory $\mathcal S$ an $m$-rigid subcategory if $\Hom_\C(\mathcal S,\mathcal S[i])=0$, $i=1,2,\cdots,m$. In this subsection, we assume that $\R$ is $2$-rigid.

\begin{lem}\label{exten}
$\A$ is extension closed.
\end{lem}

\begin{proof}
Since ${\rm Hom}_{\C}(\R,\R[2])=0$, we have $\R[1]*\R\subseteq \R*\R[1]$. Hence $\A*\A=(\RR)*(\RR)=\R*(\R[1]*\R)*\R[1]\subseteq (\R*\R)*(\R[1]*\R[1])=\RR=\A$.
\end{proof}

\begin{lem}\label{721}
${\rm Hom}_{\C}(\R,\A[1])=0$ and ${\rm Hom}_{\C}(\A,\R[2])=0$.
\end{lem}

\begin{proof}
The first equation is followed by the fact $\A[1]=\R[1]*\R[2]$ and ${\rm Hom}_{\C}(\R,\R[i])=0~(i=1,2)$. The second equation is followed by the fact $\A=\R*\R[1]$ and ${\rm Hom}_{\C}(\R[j],\R[2])=0~(j=0,1)$.
\end{proof}

\begin{lem}\label{722}
Let $A,B\in \A$. Then $[\R[1]](A,B[1])={\rm Hom}_{\C}(A,B[1])$.
\end{lem}

\begin{proof}
It is obvious that $[\R[1]](A,B[1])\subseteq {\rm Hom}_{\C}(A,B[1])$. Since $A\in\RR$, $A$ admits a triangle
$$R_{2}\to R_{1}\xrightarrow{a_1} A \xrightarrow{b_1} R_{2}[1]$$
where $R_{1},R_{2}\in \R$. Let $c:A\to B[1]$ be any morphism. By Lemma \ref{721}, $ca_1=0$. Hence there is a morphism $c_1:R_{2}[1]\to B[1]$ such that $c=c_1b_1$. Then $c\in [\R[1]](A,B[1])$. Thus $[\R[1]](A,B[1])={\rm Hom}_{\C}(A,B[1])$.
\end{proof}

By this lemma, we can get the following corollary immediately.

\begin{cor}\label{R[1]}
$\X\subseteq \A$ is $\R[1]$-rigid if and only if it is rigid.
\end{cor}

\begin{rem}
Since $\A$ is extension closed by Lemma \ref{exten}, it is an
extriangulated category {\rm (}see {\rm\cite{NP}} for details{\rm)}. Lemma \ref{721} shows that $\R$ (resp. $\R[1]$) is a subcategory of projective (resp. injective) objects in $\A$.
\end{rem}

\begin{defn}
A pair of subcategories $(\U,\V)$ in $\A$ which are closed under direct summands is called a cotorsion pair if the following conditions are satisfied:
\begin{itemize}
\item[\rm (1)] ${\rm Hom}_\C(\U,\V[1])=0$;
\item[\rm (2)] any object $A\in \A$ admits two-triangles
$$ A[-1]\to V_A\to U_A\to A~~\mbox{and}~~A\to V^A\to U^A\to A[1]$$
where $U_A,V^A\in \U$ and $V_A,V^A\in \V$.
\end{itemize}
\end{defn}

The following remark is useful.

\begin{rem}{\rm(\cite[Remark 2.2]{LN})}\label{useful}
If $(\U,\V)$ be a cotorsion pair in $\A$. Then
\begin{itemize}
\item[\rm (1)] $\U=\{A\in \A~|~{\rm Hom}_\C(A,\V[1])=0\}$;
\item[\rm (2)] $\V=\{A'\in \A~|~{\rm Hom}_\C(\U,A'[1])=0\}$;
\item[\rm (3)] $\R\subseteq \U$;
\item[\rm (4)] $\R[1]\subseteq \V$;
\item[\rm (5)] $\U,\V$ are closed under extensions and direct sums.
\end{itemize}
\end{rem}

\begin{lem}\label{co2}
Let $(\U,\V)$ be a cotorsion pair in $\A$. Then $(\overline \U,\overline \V)$ satisfies conditions {\rm (a1), (a2), (b2)} and {\rm (c1)} in Definition \ref{lcp} (hence $(\overline \U,\overline \V)$ is both a $\tau$-cotorsion pair and a left weak cotorsion pair).
\end{lem}

\begin{proof}
(c1) Let $W\in \Fac \overline \V$ be a non-zero indecomposable object. Then it admits an epimorphism $\overline v:V\to W$. Since $V$ admits a triangle
$$R\to V\to R'[1]\to R[1]$$
with $R,R'\in \R$, we get the following commutative diagram
$$\xymatrix{
R \ar[r] \ar@{=}[d] &V \ar[r] \ar[d]^v &R'[1] \ar[r] \ar[d] &R[1] \ar@{=}[d]\\
R \ar[r] &W \ar[r]^w &X \ar[r] &R[1].
}
$$
By Lemma \ref{r1}, $X\in \A$. Since $\overline w\overline v=0$, we have $\overline w=0$. Then from the following exact sequence
$$\Hom_\C(\R, W)\xrightarrow{\Hom_\C(\R,w)=0} \Hom_\C(\R,X)\to \Hom_\C(\R,R[1])=0$$
we can find that $X\in \R[1]$. Since we have the following triangle
$$V\to W\oplus R'[1] \to X\to V[1],$$
by Remark \ref{useful}, $W\in \V$.

(a1) Let $A$ be any indecomposable object in $\A \backslash \R[1]$. Since
$(\U,\V)$ is a cotorsion pair, $A$ admits a triangle $V_A\to U_A\to A\to V_A[1]$ with $V_A\in \V$ and $U_A\in \U$. Since $\Fac\overline \V=\overline \V$, we can get a short exact sequence $0\to V_A'\to U_A\to A\to 0$ in $\overline \A$ with $V_A'\in \overline \A$. By Corollary \ref{ex}, ${\rm Ext}^1_{\overline \A}(\overline \U,\overline \V)=0$. If ${\rm Ext}^1_{\overline \A}(A,\overline \V)=0$, the short exact sequence splits, hence $A$ is a direct summand of $U_A$. Then $A\in \U$, which implies that $\overline \U={^{\bot_1}}\overline \V$.

(b2) $A$ admits a triangle $R_0\to R_A\xrightarrow{r_A} A\to R_0[1]$ with $R_0,R_A\in \R$.
Since $(\U,\V)$ is a cotorsion pair, $R_A$ admits a triangle $R_A\xrightarrow{r} V_1\to U_1\to R_A[1]$ with $U_1\in \U$ and $\V_1\in \V$. Then we have the following commutative diagram.
$$\xymatrix{
R_0 \ar[r] \ar@{=}[d] &R_A \ar[r]^{r_A} \ar[d]^{r} &A \ar[r] \ar[d]^{a_1} &R_0[1] \ar@{=}[d]\\
R_0 \ar[r] &V_1 \ar[r] \ar[d] &V_2 \ar[r] \ar[d] &R_0[1]\\
&U_1 \ar[d] \ar@{=}[r] &U_1 \ar[d]\\
&R_A[1] \ar[r] &A[1]
}
$$
We can get an exact sequence $R_0\to V_1\to V_2\to 0$ by Lemma \ref{r2}, then $V_2\in \Fac \overline \V=\overline \V$. By Remark \ref{useful}, $V_2\in \V$.
$A$ admits an exact sequence $A\to V_2\to U_1\to 0$ by Lemma \ref{r2}, where $V_2\in \overline \V$ and $U_1\in\overline \U$. Since $a_1$ is a left $\V$-approximation, we have $\overline a_1$ is a left $\overline \V$-approximation.

(a2) When $A\in \R$, then $V_2\in \U\cap \V$ by Remark \ref{useful}.
\end{proof}

%\begin{cor}\label{contra}
%If $\overline \V$ is functorially finite in $\overline \A$, then $\overline \U\cap \overline \V$ satisfies condition {\rm (c2)} in Definition \ref{lcp}.
%\end{cor}
%
%\begin{proof}
%Since $\overline \V$ is functorially finite in $\overline \A$, it is enough to show that $\U\cap \V$ is contravariantly finite in $\V$. Any object $V\in \V$ admits a triangle
%$$V_1\to U_1\xrightarrow{u_1} V\to V_1[1]$$
%with $V_1\in \V$ and $U_1\in \U$. By Remark \ref{useful}, $U_1\in \U\cap \V$. Since ${\rm Hom}_{\C}(\U,\V[1])$, $u_1$ is a right $(\U\cap \V)$-approximation.
%\end{proof}

\begin{lem}\label{ext}
Let $\overline \T$ be a subcategory which is closed under extensions and $\Fac\overline \T=\overline \T$. Denote $\add(\T\cup \R[1])$ by $\widehat{\T}$. Then $\widehat{\T}$ is extension closed.
\end{lem}

\begin{proof}
Let $T_1\xrightarrow{t} T_2\to T_3\to T_1[1]$ be a triangle with $T_1,T_3\in \widehat{\T}$. Then by Lemma \ref{r2} and Lemma \ref{722}, we have an exact sequence $T_1\xrightarrow{\overline t} T_2\to T_3\to 0$ in $\overline \A$. Let $T_1\xrightarrow{\overline t_1} T\xrightarrow{\overline t_2} T_2$ be the epic-monic factorization of $\overline t$. Then $T\in \Fac \overline \T=\overline \T$. Since $\overline \T$ is extension closed, we have $T_2\in \overline \T$ in $\overline \A$. Then $T_2\in \widehat{\T}$ in $\A$.
\end{proof}

\begin{lem}\label{co}
Let $\overline \T$ be a subcategory of $\overline \A$ such that:
\begin{itemize}
\item[\rm (1)] any object $R\in \R$ admits a left $\overline \T$-approximation;
\item[\rm (2)] $\Fac\overline \T=\overline \T$;
\item[\rm (3)] $\overline \T$ is closed under extensions.
\end{itemize}
Denote $\add(\T\cup \R[1])$ by $\widehat{\T}$, then $\widehat{\T}$ admits a cotorsion pair $(\mathcal S,\widehat{\T})$ in $\A$.
\end{lem}

\begin{proof}
Denote $\{S\in \A ~|~ {\rm Hom}_{\C}(S,\widehat{\T}[1])=0\}$ by $\mathcal S$. Note that $\mathcal S$ and $\widehat{\T}$ are closed under direct summands. Let $A$ be any object in $\A$. It admits a triangle $R_0\to R_A\xrightarrow{r_A} A\to R_0[1]$ with $R_0,R_A\in \R$. If $\overline r':R_A\to T_1'$ is a left $\overline \T$-approximation, then $r'$ is a left $\widehat{\T}$ approximation in $\C$. Hence $R_A$ admits a triangle $S_1[-1]\to R_A\xrightarrow{r} T_1\to S_1$ where $r$ is a left minimal $\widehat{\T}$-approximation. By Lemma \ref{ext} and  Wakamatsu's Lemma, $S_1\in \mathcal S$. Thus we have the following commutative diagram.
$$\xymatrix{
&S_1[-1] \ar@{=}[r] \ar[d] &S_1[-1] \ar[d]\\
R_0 \ar[r] \ar@{=}[d] &R_A \ar[r]^{r_A} \ar[d]^{r} &A \ar[r] \ar[d]^{a_1} &R_0[1] \ar@{=}[d]\\
R_0 \ar[r] &T_1 \ar[r] \ar[d] &T_2 \ar[r] \ar[d] &R_0[1]\\
&S_1 \ar@{=}[r] &S_1
}
$$
Then we can get an exact sequence $R_0\to T_1\to T_2\to 0$ by Lemma \ref{r2}, which implies that $T_2\in \widehat{\T}$. Since $A$ is arbitrary, we know that $R_0$ admits a triangle $S_0[-1]\to R_0\to T_0\to S_0$ where $T_0\in \widehat{\T}$ and $S_0\in\mathcal S$. Then we have the following commutative diagram.
$$\xymatrix{
&S_0[-1] \ar@{=}[r] \ar[d] &S_0[-1] \ar[d]\\
A[-1] \ar[r] \ar@{=}[d]&R_0 \ar[r] \ar[d] &R_A \ar[r]^{r_A} \ar[d] &A \ar@{=}[d]\\
A[-1] \ar[r] &T_0 \ar[r] \ar[d] &S \ar[r] \ar[d] &A\\
&S_0 \ar@{=}[r] &S_0
}
$$
By Remark \ref{useful} $S\in \mathcal S$. By definition, $(\mathcal S,\widehat{\T})$ is a cotorsion pair in $\A$.
\end{proof}

\begin{prop}\label{tri}
Assume $\overline \T$ satisfies the conditions in Lemma \ref{co}. Then $({^{\bot_1}}{\overline \T},\overline \T)$ is both a left weak cotorsion pair and a $\tau$-cotorsion pair. Moreover, if $\overline \T$ is a functorially finite torsion class, then $({^{\bot_1}}{\overline \T},\overline \T)$ is both a left weak cotorsion torsion pair and a $\tau$-cotorsion torsion pair.
\end{prop}

\begin{proof}
This is followed by Lemma \ref{co}, Lemma \ref{co2}.
\end{proof}
\vspace{2mm}

{\bf\hspace{-4mm}6.2~~ Abelian categories of which the bounded derived categories are Krull-Schmidt.}
\hspace{2mm}
From this subsection, let $R$ be a commutative noetherian ring which is complete and local. Let $\widehat{\A}$ be an {\rm Ext}-finite abelian category over $R$ with enough projectives. Let $\mathcal P$ be the subcategory of projective objects. Let $\C=\mathrm{D}^b(\widehat{\A})$. By \cite[Corollary B]{LC}, $\C$ is Krull-Schmidt. Moreover, it is Hom-finite over $R$. Since $\widehat{\A}$ is the heart of a $t$-structure in $\C$, $\widehat{\A}$ is also Krull-Schmidt. Note that $\Hom_{\C}(\mathcal P,\mathcal P[i])=0$, $\forall i>0$. We denote $\mathcal P*\mathcal P[1]$ by $\A$ and $(\mathcal P*\mathcal P[1])/\mathcal P[1]$ by $\overline \A$.

\begin{lem}\label{732}
${\rm Hom}_{\C}(\A,\widehat{\A}[2])=0$ and ${\rm Hom}_{\C}(\A,\widehat{\A}[1])=[\mathcal P[1]](\A,\widehat{\A}[1])$.
\end{lem}

\begin{proof}
Let $A$ be any object in $\A$. Then $A$ admits a triangle $P_1\to P_0\xrightarrow{p_0} A\xrightarrow{a} P_1[1]$ with $P_0,P_1\in \mathcal P$.

Let $a_0:A\to \widehat{A}[2]$ be a morphism with $\widehat{A}\in \widehat{\A}$. Then $a_0p_0=0$ and $a_0$ factors through $P_1[1]$. But $\Hom_{\C}(P_1[1],\widehat{\A}[2])=0$. Hence $a_0=0$.

Let $a_1:A\to \widehat{A}[1]$ be a morphism with $\widehat{A}\in \widehat{\A}$. Then $a_1p_0=0$ and $a$ factors through $P_1[1]$. Hence ${\rm Hom}_{\C}(\A,\widehat{\A}[1])=[\mathcal P[1]](\A,\widehat{\A}[1])$.
\end{proof}

\begin{prop}\label{733}
We have an equivalence of additive categories: $\overline \A\simeq \widehat{\A}$.
\end{prop}

\begin{proof}
Let $A\in \A$ be an object having no direct summand in $\mathcal P[1]$. $A$ admits a triangle $P_1\xrightarrow{p} P_0\to A\to P_1[1]$ with $P_0,P_1\in \mathcal P$. Let $P_1\xrightarrow{p_1} A_1''\xrightarrow{a_1} P_0$ be the epic-monic factorization of $p$ in $\widehat{\A}$. Then we have two short exact sequences
$$0\to A_1''\xrightarrow{a_1} P_0\to \widehat{A}_0\to 0, ~0\to A_2''\to P_1\xrightarrow{p_1} A_1''\to 0 \text{ in }\widehat{\A}$$
which induces two triangles
$$A_1''\xrightarrow{a_1} P_0\to \widehat{A}_0\to A_1''[1], ~A_2''\to P_1\xrightarrow{p_1} A_1''\to A_2''[1] \text{ in }\C.$$
Then we have the following commutative diagram of triangles
$$\xymatrix{
P_1 \ar[r]^{p_1} \ar@{=}[d] &A_1'' \ar[r] \ar[d]^{a_1} &A_2''[1] \ar[r] \ar[d] &P_1[1] \ar@{=}[d]\\
P_1 \ar[r]^{p} &P_0 \ar[r] \ar[d] &A \ar[r] \ar[d]^{a_0} &P[1]\\
&\widehat{A}_0 \ar@{=}[r] \ar[d] &\widehat{A}_0 \ar[d]\\
&A_1''[1] \ar[r] &A_2''[2]
}$$
Since $\Hom_{\C}(\widehat{\A}[1],\widehat{\A})=0$, $a_0$ is a left $\widehat{\A}$-approximation. Hence $A$ admits a triangle
$$A_1[1]\to A\xrightarrow{a} \widehat{A}\to A_1[2]$$
where $a$ is a left minimal $\widehat{\A}$-approximation and $A_1\in \widehat{\A}$. Then for any two object $A,A'\in \overline \A$ and a morphism $\overline f:A\to A'$, we have the following commutative diagram of triangles
$$\xymatrix{
A_1[1] \ar[r] \ar[d] &A \ar[r]^{a} \ar[d]^f &\widehat{A} \ar[r]^{\widehat{a}} \ar[d]^{\hat{f}} &A_1[2] \ar[d]\\
A_1'[1] \ar[r] &A' \ar[r]^{a'}  &\widehat{A}' \ar[r]^{\widehat{a}'}  &A_1'[2].
}$$
Then we can define a functor $\mathbb{G}:\overline \A\to \widehat{\A}$ such that
$$\mathbb{G}(A)=\widehat{A},\quad \mathbb{G}(\overline f)=\hat{f}.$$
It is well-defined: if there is a morphism $f_1:A\to A'$ such that $\overline f_1=\overline f$, then we have the following commutative diagram of triangles
$$\xymatrix{
A_1[1] \ar[r] \ar[d] &A \ar[r]^{a} \ar[d]^{f_1} &\widehat{A} \ar[r] \ar[d]^{\hat{f}_1} &A_1[2] \ar[d]\\
A_1'[1] \ar[r] &A' \ar[r]^{a'}  &\widehat{A}' \ar[r]  &A_1'[2].
}$$
Since $f_1-f$ factors through $\mathcal P[1]$, $0=a'(f_1-f)=(\hat{f}_1-\hat{f})a$. Then $\hat{f}_1-\hat{f}$ factors through $A_1[2]$, but $\Hom_{\C}(\widehat{\A}[2],\widehat{\A})=0$, hence $\hat{f}_1=\hat{f}$. By definition, $\mathbb{G}$ is additive. We show $\mathbb{G}$ is fully-faithful and dense.

$\mathbb{G}$ is full: for any morphism $\hat{g}:\widehat{A}\to \widehat{A}'$, since $\hat{a'}\hat{g}a=0$ by Lemma \ref{732}, we have a commutative diagram of triangles
$$\xymatrix{
A_1[1] \ar[r] \ar[d] &A \ar[r]^{a} \ar[d]^{g} &\widehat{A} \ar[r]^{\hat{a}} \ar[d]^{\hat{g}} &A_1[2] \ar[d]\\
A_1'[1] \ar[r] &A' \ar[r]^{a'}  &\widehat{A}' \ar[r]^{\hat{a}'}  &A_1'[2]
}$$
which implies $\mathbb{G}(g)=\hat{g}$.

$\mathbb{G}$ is faithful: if $\mathbb{G}(\overline f)=0$, then $a'f=0$. Hence $f$ factors through $A_1'[1]$. By Lemma \ref{732}, $\overline f=0$.

$\mathbb{G}$ is dense: let $\widehat{A}_0$ be any object in $\widehat{\A}$. It admits two show exact sequences:
$$0\to A_1''\xrightarrow{a_1} P_0\to \widehat{A}_0\to 0, ~ 0\to A_2''\to P_1\xrightarrow{p_1} A_1''\to 0$$
in $\widehat{\A}$ with $P_0,P_1\in \mathcal P$. Then we have two triangles
$$A_1''\xrightarrow{a_1} P_0\to \widehat{A}_0\to A_1''[1], ~A_2''\to P_1\xrightarrow{p_1} A_1''\to A_2[1] \text{ in }\C.$$
We can get the following commutative diagram of triangles
$$\xymatrix{
P_1 \ar[r]^{p_1} \ar@{=}[d] &A_1'' \ar[r] \ar[d]^{a_1} &A_2''[1] \ar[r] \ar[d] &P_1[1] \ar@{=}[d]\\
P_1 \ar[r]^{p} &P_0 \ar[r] \ar[d] &A_0' \ar[r] \ar[d]^{a_0} &P_1[1]\\
&\widehat{A}_0 \ar@{=}[r] \ar[d] &\widehat{A}_0 \ar[d]\\
&A_1''[1] \ar[r] &A_2''[2]
}$$
where $A_0'\in \A$. We can assume $A_0'$ has no direct summand in $\widehat{\A}[1]$. $A_0'$ admits a triangle
$$A_1'[1] \to A_0'\xrightarrow{a_0'} \widehat{A}_0'\to A'_1[2]$$
where $A_1'\in \widehat{\A}$, $a_0'$ is a left minimal $\widehat{A}$-approximation. Since $a_0$ is a left $\widehat{\A}$-approximation, we have the following commutative diagram.
$$\xymatrix{
A_1''[1] \ar[r] \ar[d] &A_0' \ar[r]^{a_0} \ar@{=}[d]  &\widehat{A}_0 \ar[r] \ar[d]  &A_1''[2] \ar[d]\\
A_1'[1] \ar[r] \ar[d] &A_0' \ar[r]^{a_0'} \ar@{=}[d] &\widehat{A}_0' \ar[r] \ar[d] &A_1'[2] \ar[d]\\
A_1''[1] \ar[r] &A_0' \ar[r]^{a_0}  &\widehat{A}_0 \ar[r]  &A_1''[2]
}$$
Hence $\widehat{A}_0$ is a direct summand of $\widehat{A}_0'\oplus A_1''[2]$ and $\widehat{A}_0'$ is a direct summand of $\widehat{A}_0$. But $\widehat{\A}[2]\cap\widehat{\A}=0$, we have $\widehat{A}_0\cong \widehat{A}_0'$.
\end{proof}

We introduce the following notions:
\begin{itemize}
\item[(1)] $\widehat{\A}_{c\text{-}s\tau\text{-}til}=:\{\text{ contravariantly finite support }\tau\text{-tilting subcategories in }\widehat{\A}~\}$;
\item[(2)] $\widehat{\A}_{lw\text{-}ctp}=:\{\text{ left weak cotorsion torsion pairs in }\widehat{\A}~\}$;
\item[(3)] $\widehat{\A}_{\tau\text{-}ctp}=:\{~\tau\text{-cotorsion torsion pairs in }\widehat{\A}~\}$;
\item[(4)] $\widehat{\A}_{f\text{-}tor}=:\{\text{ functorially finite torsion class } \T\subseteq \widehat{\A } \}$.
\end{itemize}

\begin{thm}\label{main722}
There are bijections
$$ \widehat{\A}_{c\text{-}s\tau\text{-}til} \longleftrightarrow \widehat{\A}_{f\text{-}tor} \longleftrightarrow \widehat{\A}_{lw\text{-}ctp}$$
given by
\begin{itemize}
\item[\rm (a1)] $\widehat{\A}_{c\text{-}s\tau\text{-}til}\ni\M\mapsto \Fac \M\in \widehat{\A}_{f\text{-}tor}$;
\item[\rm (a2)] $\widehat{\A}_{f\text{-}tor}\ni\T\mapsto {^{\bot_1}}\T\cap\T\in \widehat{\A}_{c\text{-}s\tau\text{-}til}$;
\item[\rm (b1)] $\widehat{\A}_{f\text{-}tor}\ni\T\mapsto ({^{\bot_1}}\T,\T)\in \widehat{\A}_{lw\text{-}ctp}$;
\item[\rm (b2)] $\widehat{\A}_{lw\text{-}ctp}\ni({\mathcal S}, \T)\mapsto \T\in  \widehat{\A}_{f\text{-}tor}$.
\end{itemize}
\end{thm}

\begin{proof}
(b1) and (b2): This is followed by Proposition \ref{733} and Proposition \ref{tri}.

(a1) and (a2):  By Proposition \ref{733} and Proposition \ref{tri}, if $\T$ is a functorially finite torsion class, then $({^{\bot_1}}\T,\T)$ is both a $\tau$-cotorsion torsion pair and a left weak cotorsion torsion pair. Then by Theorem \ref{PZZ}, the bijections in (a1) and (a2) are induced by the bijections in (b1) and (b2).
\end{proof}

By this theorem and Theorem \ref{PZZ}, we can get the following corollary, which generalizes \cite[Theorem 4.10]{AST}.

\begin{cor}\label{main722-c}
$\widehat{\A}_{\tau\text{-}ctp}=\widehat{\A}_{lw\text{-}ctp}$.
\end{cor}

%\begin{rem}
%By {\rm \cite[Theorem 3.8]{PZZ}} and Lemma \ref{co}, if $(\U,\V)$ satisfies conditions {\rm (a1), (a2)} and {\rm (c1)} in Definition \ref{lcp} (which means it is a $\tau$-cotorsion pair such that $\Fac \V=\V$), it is also a left weak cotorsion pair.
%\end{rem}

\begin{lem}\label{lefta}
Let $\widehat{\X}$ be a subcategory of $\widehat{\A}$ such that every projective object admits a left $\widehat{\X}$-approximation. Let $\mathbb{F}$ be the quasi-inverse of $\mathbb{G}$. Then every object in $\mathcal P$ admits a left $\mathbb{F}(\widehat{\X})$-approximation in $\A$.
\end{lem}

\begin{proof}
Let $P\in\mathcal P$ be any projective object. Denote $\mathbb{F}(\widehat{\X})$ by $\overline \X$, assume that $\X\cap \mathcal P[1]=0$.

Let $P\xrightarrow{\hat{p}} \widehat{X}$ be a left $\widehat{\X}$-approximation of $P$. By the proof of Proposition \ref{733}, $\widehat{X}$ admits a triangle
$$A_1[1]\to X\xrightarrow{x} \widehat{X}\xrightarrow{\hat{x}} A_1[2]$$
where $A_1\in \widehat{\A}$, $X\in \A$ and $X\cong \mathbb{F}(\widehat{X})$. Since $\hat{x}\hat{p}=0$, there is a morphism $p:P\to X$ such that $xp=\hat{p}$. We show that $p$ is a left $\X$-approximation.

Let $X'\in \X$ and $p':P\to X'$ be any morphism. $X'$ admits a triangle
$$A_1'[1]\to X'\xrightarrow{x'} \widehat{X}'\xrightarrow{\hat{x}'} A_1'[2]$$
where $A_1'\in \widehat{\A}$ and $\widehat{X}'\in \widehat{\X}$. Since $\hat{p}$ is a left $\widehat{\X}$-approximation, there is a morphism $\hat{y}:\widehat{X}\to \widehat{X}'$ such that $\hat{y}\hat{p}=x'p'$. Since $\Hom_{\C}(X,A_1'[2])=0$, there is a morphism $y:X\to X'$ such that $x'y=\hat{y}x$. Hence
$$x'p'=\hat{y}\hat{p}=\hat{y}xp=x'yp.$$
Then $p'-yp$ factors through $A_1'[1]$, which implies $p'=yp$.
\end{proof}

{\bf\hspace{-4mm}6.3~~ $\tau$-rigid pairs}
\hspace{2mm}

In this subsection, we always assume that $(\widehat{\X},\E)$ is a $\tau$-rigid pair satisfying the following conditions
\begin{itemize}
\item[(X0)] $\widehat{\X}$ is not support $\tau$-tilting;
\item[(X1)] $\widehat{\X}$ is contravariantly finite;
\item[(X2)] every projective object admits a left $\widehat{\X}$-approximation;
\item[(X3)] $\E=\{P\in \mathcal P~|~\Hom_{\widehat{\A}}(P,\widehat{\X})=0\}$.
\end{itemize}

\begin{lem}\label{conf-1}
$\Fac \widehat{\X}$ is a functorially finite torsion class.
\end{lem}

\begin{proof}
This is followed by Proposition \ref{733} and Lemma \ref{conf}.
\end{proof}

\begin{rem}
By Lemma \ref{lefta}, Proposition \ref{facmax}, Theorem \ref{main}, $(\widehat{\X},\E)$ is contained in a support $\tau$-tilting subcategory $(\widehat{\M},\E)$ such that $\Fac\widehat{\M}=\Fac\widehat{\X}$.
\end{rem}

We first show a technical lemma.

\begin{lem}\label{tech}
If $\E=0$ and every projective object $P\in \mathcal P$ admits a short exact sequence
$$0\to P\xrightarrow{\hat{x}} \widehat{X}\to \widehat{M}\to 0$$
where $\hat{x}$ is a left $\widehat{X}$-approximation, then $(\widehat{\X},0)$ is contained in a support $\tau$-tilting pair $(\widehat{\mathcal N},0)$ such that every projective object $P\in \mathcal P$ admits an exact sequence
$$P\xrightarrow{\hat{n}} \widehat{N}\to \widehat{X}'\to 0$$
where $\widehat{X}'\in \widehat{\X}$ and $\hat{n}$ is a left $\widehat{\N}$-approximation.
\end{lem}

\begin{proof}
Denote $\mathbb{F}(\widehat{\X})$ by $\overline \X$. Assume $\X\cap \R[1]=0$. By Proposition \ref{733} and Theorem \ref{thm1}, $\X$ is two-term $\mathcal P[1]$-rigid in $\C$. We show that any object $P[1]$ admits a right $\X$-approximation, then by Lemma \ref{lefta} $\X$ is $\mathcal P[1]$-functorially finite. Hence by Theorem \ref{thm1} and Proposition \ref{733}, $(\widehat{\X},0)$ is contained in a support $\tau$-tilting pair $(\widehat{\mathcal N},0)$ such that every projective object $P\in \mathcal P$ admits an exact sequence
$$P\xrightarrow{\hat{n}} \widehat{N}\to \widehat{X}'\to 0$$
where $\widehat{X}'\in \widehat{\X}$ and $\hat{n}$ is a left $\widehat{\N}$-approximation. In fact, $\widehat{\N}=\mathbb{G}(\overline \N_\X)$.

By assumption, every projective object $P$ admits a short exact sequence
$$0\to P\xrightarrow{\hat{x}} \widehat{X}\to \widehat{M}\to 0$$
where $\hat{x}$ is a left $\widehat{X}$-approximation. By the proof of Proposition \ref{733}, $\widehat{M}$ admits a triangle
$$A_0[1] \to M\xrightarrow{m} \widehat{M}\to A_0[2]$$
where $A_0\in \widehat{\A}$, $M\in \A$. We can assume that $M$ has no direct summand in $\widehat{\A}[1]$. Then we have the following commutative diagram of triangles
$$\xymatrix{
&A_0[1] \ar@{=}[r] \ar[d] &A_0[1] \ar[d]\\
P \ar[r] \ar@{=}[d] &X' \ar[r] \ar[d] &M \ar[r]^p \ar[d]^m &P[1] \ar@{=}[d]\\
P \ar[r]^{\hat{x}} &\widehat{X} \ar[r] \ar[d] &\widehat{M} \ar[r]^{\hat{p}} \ar[d] &P[1]\\
&A_0[2] \ar@{=}[r] &A_0[2]
}
$$
where $X'\in \A$ by Lemma \ref{r1}. Note that $X'$ has no direct summand in $\widehat{\A}[1]$, otherwise $M$ has a direct summand in $\widehat{\A}[1]$. Since $\widehat{X}$ admits a triangle
$$A_2[1] \to X \to \widehat{X}\to A_2[2]$$
where $X\in \A$ and $X$ has no direct summand in $\widehat{\A}[1]$, we can get the following commutative diagram
$$\xymatrix{
A_0[1] \ar[r] \ar[d] &X' \ar[r] \ar[d] & \widehat{X} \ar[r] \ar@{=}[d] &A_0[2] \ar[d]\\
A_2[1] \ar[r] \ar[d] &X \ar[r] \ar[d] & \widehat{X} \ar[r] \ar@{=}[d] &A_2[2] \ar[d]\\
A_0[1] \ar[r] &X' \ar[r]  & \widehat{X} \ar[r]  &A_0[2]\\
}
$$
which implies that $X'$ is a direct summand of $X\oplus A_0[1]$ and $X$ is a direct summand of $X'\oplus A_2[1]$. But $X$ and $X'$ have no direct summand in $\widehat{\A}[1]$, we have $X'\cong X\cong \mathbb{F}(\widehat{X})$.
Let $\hat{s}:\widehat{X}_1\to \widehat{M}$ be a right $\widehat{\X}$-approximation. Since $\widehat{X}_1$ admits a triangle
$$A_1[1]\to X_1\xrightarrow{x_1} \widehat{X}_1\xrightarrow{\hat{x}_1} A_1[2]$$
where $A_1\in \widehat{\A}$, $X_1\in \A$, we have a commutative diagram
$$\xymatrix{
X_1 \ar[r]^s \ar[d]_{x_1} &M \ar[d]^m \ar[r]^p &P[1] \ar@{=}[d]\\
\widehat{X_1} \ar[r]^{\hat{s}} &\widehat{M} \ar[r]^{\hat{p}} &P[1].
}
$$
We show that $ps$ is a right $\X$-approximation of $P[1]$.

Let $t:X'_1\to P[1]$ be any morphism with $X'_1\in \X$. Since $\X$ is two-term $\mathcal P[1]$-rigid in $\C$, we have $[\mathcal P[1]](X'_1,X'[1])=0$. Then there is a morphism $s':X'_1\to M$ such that $t=ps'$. Since $X'_1$ admits a triangle
$$A_1'[1]\to X'_1\xrightarrow{x'_1} \widehat{X}'_1\xrightarrow{\hat{x}'_1} A_1'[2]$$
where $A_1'\in \widehat{\A}$ and $\widehat{X}'_1\in \widehat{\X}$, we have a commutative diagram
$$\xymatrix{
X'_1 \ar[r]^{s'} \ar[d]_{x'_1} &M \ar[d]^m \ar[r]^p &P[1] \ar@{=}[d]\\
\widehat{X}'_1 \ar[r]^{\hat{s}'} &\widehat{M} \ar[r]^{\hat{p}} &P[1].
}
$$
Since $\hat{s}:\widehat{X}_1\to \widehat{M}$ is a right $\widehat{\X}$-approximation, there is a morphism $\hat{f}:\widehat{X}'_1\to \widehat{X}$ such that $\hat{s}'=\hat{s}\hat{f}$. Then we have a commutative diagram
$$\xymatrix{
X'_1 \ar[r]^{f} \ar[d]_{x'_1} &X_1 \ar[d]^{x_1} \\
\widehat{X}'_1 \ar[r]^{\hat{f}} &\widehat{X}_1
}
$$
Hence
$$t=ps'=\hat{p}ms'=\hat{p}\hat{s}'x'_1=\hat{p}\hat{s}\hat{f}x'_1=\hat{p}\hat{s}x_1f=psf.$$
\end{proof}

By \cite[Lemma 4.2 and Theorem 4.4]{PZZ}, we have the following lemma.

\begin{lem}\label{lPZZ}
$\E^{\bot}$ is a Serre subcategory of $\widehat{A}$ with enough projectives given by
$$\mathcal P_{\widehat{\X}}=\add \{K ~|~ \exists P\in \mathcal P \text{ and a left }\widehat{\X}\text{-approximation } P\xrightarrow{f} \widehat{X} \text{ such that }K={\rm Im}f \}.$$
Moreover, $\widehat{\X}$ is a partial tilting subcategory in $\E^{\bot}$.
\end{lem}

\begin{proof}
Since $\E\subseteq \mathcal P$, $\E^{\bot}$ is closed under factors, subobjects and extensions. Hence $\E^{\bot}$ is a Serre subcategory of $\widehat{A}$. By \cite[Lemma 4.2]{PZZ}, $\E^{\bot}$ has enough projectives $\mathcal P_{\widehat{\X}}$. The ``Moreover" part is followed by \cite[Theorem 4.4]{PZZ}.
\end{proof}

Now we can show the following proposition.

\begin{prop}\label{PN}
$(\widehat{\X},\E)$ is contained in a support $\tau$-tilting pair $(\widehat{\mathcal N},\E)$ such that every projective object $P\in \mathcal P$ admits an exact sequence
$$P\xrightarrow{\hat{n}} \widehat{N}\to \widehat{X}\to 0$$
where $\widehat{X}\in \widehat{\X}$ and $\hat{n}$ is a left $\widehat{\N}$-approximation.
\end{prop}

\begin{proof}
For convenience, denote $\E^{\bot}$ by $\B$. By Lemma \ref{lPZZ}, $\B$ is a Serre subcategory of $\widehat{A}$ with enough projectives and $\widehat{\X}$ is partial tilting in $\B$. Moreover, any non-zero indecomposable object $K_0\in \mathcal P_{\widehat{\X}}$ admits a commutative diagram
$$\xymatrix{
P_1\ar[r]^{f_1} \ar[dr]_{g_1} &\widehat{X}_1\\
K_0 \ar[r]_{\kappa} &K_1 \ar[u]_{k_1}
}
$$
where $f_1$ is a left $\widehat{\X}$-approximation, $g_1$ is epic, $k_1$ is monic, $K_1={\rm Im} f_1$ and $\kappa$ is a section. Then we can get that $k_1\kappa$ is a left $\widehat{\X}$-approximation and a monomorphism. Then by Lemma \ref{tech}, $\widehat{\X}$ is contained in a support $\tau$-tilting subcategory $\widehat{\N}$ in $\B$ such that any object $K\in \mathcal P_{\widehat{\X}}$ admits an exact sequence
$$K\xrightarrow{\hat{k}} \widehat{N}\to \widehat{X}\to 0$$
where $\widehat{X}\in \widehat{\X}$ and $\hat{k}$ is a left $\widehat{\N}$-approximation. Since $\B$ is a Serre subcategory, $\Ext^1_{\B}(\widehat{\N},\Fac \widehat{\N})=0$ implies that $\Ext^1_{\widehat{\A}}(\widehat{\N},\Fac \widehat{\N})=0$. Hence $\widehat{\N}$ is $\tau$-rigid in $\A$. We show that $\widehat{\N}$ is a support $\tau$-tilting subcategory in $\widehat{\A}$. Since $\widehat{\N}\subseteq \E^{\bot}$, we show that any indecomposable object $P\in \mathcal P \backslash \E$ admits an exact sequence
$$P\xrightarrow{\hat{n}} \widehat{N}\to \widehat{X}\to 0$$
where $\widehat{X}\in \widehat{\X}$ and $\hat{n}$ is a left $\widehat{\N}$-approximation.

Since $P\in \mathcal P \backslash \E$ admits a commutative diagram
$$\xymatrix{
P\ar[rr]^{f} \ar[dr]_{g} &&\widehat{X}'\\
 &K \ar[ur]_{k}
}
$$
where $f\neq 0$ is a left $\widehat{\X}$-approximation, $g$ is epic, $k$ is monic, $K={\rm Im} f$. By the argument above, $K$ admits an exact sequence
$$K\xrightarrow{\hat{k}} \widehat{N}\xrightarrow{\widehat{x}} \widehat{X}\to 0$$
where $\widehat{X}\in \widehat{\X}$ and $\hat{k}$ is a left $\widehat{\N}$-approximation. We show that
$$P\xrightarrow{\hat{k}g=\hat{n}} \widehat{N}\to \widehat{X}\to 0$$
is the exact sequence we need. In fact, we only need to show that $\hat{n}$ is a left $\widehat{\N}$-approximation.

Since $\hat{k}$ is a left $\widehat{\N}$-approximation, there is a morphism $n: \widehat{N}\to \widehat{X}'$ such that $k=n\hat{k}$.
Let $\widehat{N}'_0$ be an indecomposable object in $\widehat{\N}\backslash \widehat{\X}$. Then by Lemma \ref{tech}, it admits an exact sequence
$$K_0\xrightarrow{\hat{k}_0} \widehat{N}_0\xrightarrow{\widehat{x}_0} \widehat{X}_0\to 0$$
where $K_0\in \mathcal P_{\widehat{\X}}$, $\widehat{X}_0\in \widehat{\X}$, $\hat{k}_0$ is a left $\widehat{\N}_0$-approximation and $\widehat{N}_0'$ is a direct summand of $\widehat{N}_0$. Let $\hat{n}_0:P\to \widehat{N}_0$ be any morphism, we show it factors through $\hat{n}$. For convenience, we can assume that $K_0$ admits a commutative diagram
$$\xymatrix{
P_0\ar[rr]^{f_0} \ar[dr]_{g_0} &&\widehat{X}'_0\\
 &K_0 \ar[ur]_{k_0}
}
$$
where $P_0\in \mathcal P$, $f_0\neq 0$ is a left $\widehat{\X}$-approximation, $g_0$ is epic, $k_0$ is monic, $K_0={\rm Im} f_0$.

Since $f$ is a left $\widehat{\X}$-approximation, there is a morphism $x_1:\widehat{X}'\to X_0$ such that $\hat{x}_0\hat{n}_0=x_1f$. Since $\hat{k}$ is a left $\widehat{\N}$-approximation, there is a morphism $n_1:\widehat{N}\to \widehat{X}_0$ such that $x_1k=n_1\hat{k}$. Since $K_0\in \mathcal P_{\widehat{\X}}$, there is a morphism $k_0:K\to \widehat{N}_0$ such that $\hat{x}_0k_0=n_1\hat{k}$. Since $\hat{k}$ is a left $\widehat{\N}$-approximation, there is a morphism $n_2:\widehat{N}\to \widehat{N}_0$ such that $k_0=n_2\hat{k}$. Then $$\hat{x}_0(\hat{n}_0-n_2\hat{n})=x_1f-\hat{x}_0n_2\hat{k}g=x_1f-n_1\hat{k}g=x_1f-x_1kg=0.$$
Let $l_0:L_0\to \widehat{N}_0$ be the kernel of $\widehat{x}_0$. We have a commutative diagram
$$\xymatrix{
K_0 \ar[rr]^-{\hat{k}_0} \ar[dr]^{l_1} &&\widehat{N}_0 \ar[r]^{\hat{x}_0} &\widehat{X}_0.\\
&L_0 \ar[ur]^{l_0}
}
$$
Then there is a morphism $l:P\to L_0$ such that $l_0l=\hat{n}_0-n_2\hat{n}$. Since $P$ is projective, there is a morphism $p_0:P\to P_0$ such that $l_1g_0p_0=l$. Since $f$ is a left $\widehat{\X}$-approximation, there is a morphism $x_0:\widehat{X}'\to \widehat{X}'_0$ such that $f_0p_0=x_0f$. Then there is a morphism $k_1: K\to K_0$ such that $k_0k_1=x_0k$ and $g_0p_0=k_1g$. Since $\hat{k}$ is a left $\widehat{\N}$-approximation, there is a morphism $n_3:\widehat{N}\to \widehat{N}_0$ such that $\hat{k}_0k_1=n_3\hat{k}$. Hence
$$\hat{n}_0-n_2\hat{n}=l_0l=l_0l_1g_0p_0=\hat{k}_0k_1g=n_3\hat{k}g=n_3\hat{n},$$
which implies that $\hat{n}_0=(n_2+n_3)\hat{n}$. Thus $\widehat{\N}$ is a support $\tau$-tilting subcategory. Since $\widehat{\N}\in \E^{\bot}$, $\Hom_{\widehat{\A}}(\E,\widehat{\N})=0$. Moreover,
$$\{P\in \mathcal P~|~ \Hom_{\widehat{\A}}(P,\widehat{\N})=0\}\subseteq \{P\in \mathcal P~|~ \Hom_{\widehat{\A}}(P,\widehat{\X})=0\}=\E,$$
by definition $(\widehat{\N},\E)$ is a support $\tau$-tilting pair.
\end{proof}

%By Theorem \ref{thm1} and Theorem \ref{main}, we have the following corollary.
%
%\begin{cor}
%We have $\Fac\widehat{\N}={^{\bot}}(\tau\widehat{\X})\cap \E^{\bot}$.
%\end{cor}

By Theorem \ref{thm1}, Theorem \ref{main}, Proposition \ref{733} and Proposition \ref{PN}, we can get the following theorem.

\begin{thm}\label{last}
Let $R$ be a commutative noetherian ring which is complete and local. Let $\widehat{\A}$ be an {\rm Ext}-finite abelian category over $R$ with enough projectives. Let $\mathcal P$ be the subcategory of projective objects. Assume that $(\widehat{\X},\E)$ is a $\tau$-rigid pair satisfying the following conditions:
\begin{itemize}
\item[(X0)] $\widehat{\X}$ is not support $\tau$-tilting;
\item[(X1)] $\widehat{\X}$ is contravariantly finite;
\item[(X2)] every projective object admits a left $\widehat{\X}$-approximation;
\item[(X3)] $\E=\{P\in \mathcal P~|~\Hom_{\widehat{\A}}(P,\widehat{\X})=0\}$.
\end{itemize}
Then it is contained in two support $\tau$-tilting pairs $(\widehat{\M},\E)$ and $(\widehat{\N}, \E)$ such that
$$\Fac\widehat{\M}=\Fac\widehat{\X}, \quad \Fac\widehat{\N}={^\perp(\tau\widehat{\X})}\cap{{\E}^{\perp}} \text{ and } \Fac \widehat{\M}\subsetneq \Fac \widehat{\N}.$$
Moreover, if $(\widehat{\X},\E)$ is an almost complete support $\tau$-tilting pair, then $(\widehat{\M},\E)$ and $(\widehat{\N},\E)$ are the only support $\tau$-tilting pairs which contain $(\widehat{\X},\E)$.
\end{thm}

In $\widehat{\A}$, for two subcategories $\B_1,\B_2$, let
$$\B_1*\B_2=:\{A\in \widehat{\A}~|~ \exists\text{ exact sequence } 0\to B_1\to A\to B_2\to 0,~B_i\in \B_i,i=1,2\}.$$

As an application of our results, we show the following proposition, which is a categorical version of left Bongartz completion (see \cite[Proposition 3.4]{CWZ} when $\widehat{\A}$ is a module category).

\begin{prop}\label{CWZ2}
Let $\mathcal L$ be a support $\tau$-tilting subcategory in $\widehat{\A}$. Denote $\Fac\widehat{\X}*\Fac \mathcal L$ by $\T$. If $\widehat{\mathcal N} \geq \mathcal L$, then $({^{\bot_1}}\T,\T)$ is a $\tau$-cotorsion pair. Denote ${^{\bot_1}}\T\cap\T$ by $\mathcal L^-$ and $\{P\in \mathcal P ~|~ \Hom_{\widehat{\A}}(P,\mathcal L^-)=0 \}$ by $\mathcal Q^-$. Then $(\mathcal L^-,\mathcal Q^-)$ is a support $\tau$-tilting pair that contains $(\widehat{\X},\E)$.
\end{prop}

\begin{proof}
By Lemma \ref{conf}, Proposition \ref{tri} and Proposition \ref{733}, we need to check that $\T$ satisfies the following conditions:
\begin{itemize}
\item[\rm (1)] any object $P\in \mathcal P$ admits a left $\T$-approximation;
\item[\rm (2)] $\Fac \T=\T$;
\item[\rm (3)] $\T$ is closed under extensions.
\end{itemize}

Now we give the proof for the above three conditions.
\medskip

(1) Let $P\in \mathcal P$. By Theorem \ref{PZZ}, $P$ admits an exact sequence
$$P\xrightarrow{k} K\to J\to 0$$
where $K\in \Fac\mathcal L\cap {^{\bot_1}}(\Fac\mathcal L)$, $J\in {^{\bot_1}}(\Fac\mathcal L)$ and $k$ is a left $(\Fac\mathcal L)$-approximation. By Lemma \ref{conf-1} and Theorem \ref{main722}, $({^{\bot_1}}(\Fac\widehat{\X}),\Fac\widehat{\X})$ is a left weak cotorsion pair. Hence $K$ admits a short exact sequence
$$0\to W_K\to T_K\xrightarrow{t_0} K\to 0$$
where $W_K\in \Fac \widehat{\X}$ and $T_K\in {^{\bot_1}}(\Fac\widehat{\X})$. Note that $T_K\in \T$ and there is a morphism $p:P\to T_K$ such that $t_0p=k$. By Theorem \ref{main722} and Corollary \ref{main722-c}, $P$ also admits an exact sequence
$$P\xrightarrow{w} W_P\to U_P\to 0$$
where $w$ is a left $(\Fac\widehat{\X})$-approximation. We show that $P\xrightarrow{\svecv{p}{w}} T_K\oplus W_P$ is a left $\T$-approximation. Let $T\in \T$ and $t\in \Hom_{\widehat{\A}}(P,T)$. $T$ admits a short exact sequence
$$0\to W_T\xrightarrow{w_0} T\xrightarrow{t_1} K_T\to 0$$
with $W_T\in \Fac\widehat{\X}$ and $K_T\in \Fac \mathcal L$. Since $k$ is a left $(\Fac\mathcal L)$-approximation, we have the following commutative diagram.
$$\xymatrix{
&&P \ar[r]^k \ar[d]^t &K \ar[r] \ar[d]^{k_1} &J \ar[r] &0\\
0\ar[r] &W_T \ar[r]^{w_0} &T \ar[r]^{t_1} &K_T \ar[r] &0
}$$
Since $T_K\in {^{\bot_1}}(\Fac\widehat{\X})$, there is a morphism $w_1:T_K\to T$ such that $k_1t_0=t_1w_1$. Hence $t_1t=k_1k=k_1t_0p=t_1w_1p$, which implies that there is a morphism $p_0:P\to W_T$ such that $w_0p_0=t-w_1p$. Since $p_0$ factors through $w$, we can find that $t$ factors through $P\xrightarrow{\svecv{p}{w}} T_K\oplus W_P$.

(2) Let $S\in \Fac \T$. It admits an epimorphism $t_0:T_S\to S$ with $T_S\in \T$. $T_S$ admits a short exact sequence
$$0\to W\to T_S\to K\to 0$$
with $W\in \Fac\widehat{\X}$ and $K\in \Fac \mathcal L$. $S$ admits a short exact sequence
$$0\to W_S\to S\xrightarrow{s} Y_S\to 0$$
with $W_S\in \Fac\widehat{\X}$ and $Y_S\in \widehat{\X}^{\bot}$. Then we have a commutative diagram
$$\xymatrix{
0 \ar[r] &W \ar[r] \ar@{.>}[d] &T_S \ar[d]^{t_0} \ar[r] &K \ar[r] \ar@{.>}[d]^y &0\\
0 \ar[r] &W_S \ar[r]  &S  \ar[r]^s &Y_S \ar[r] &0.
}
$$
Since $st_0$ is an epimorphism, so is $y$. Hence $Y_S\in \Fac \mathcal L$, which implies that $S\in \T$.

(3) Let $T\in \T*\T=(\Fac \widehat{\X}*(\Fac\mathcal L*\Fac \widehat{\X}))*\Fac \mathcal L$. Then we have the following short exact sequences
$$0\to A_1\to T\to K_1\to 0, \quad 0\to W_1\to A_1\to B_1\to 0,\quad 0\to K_2\to B_1\to W_2\to 0$$
with $K_1,K_2\in \Fac\mathcal L$ and $W_1,W_2\in \Fac \widehat{\X}$. Since $K_2\in \Fac\mathcal L$, we have $\Ext^1_{\widehat{\A}}(\widehat{\X},K_2)=0$. Since $W_2$ admits an epimorphism $x_2:X_2\to W_2$ with $X_2\in \widehat{\X}$, we have the following commutative diagram
$$\xymatrix{
0 \ar[r] &K_2 \ar[r]^-{\svecv{1}{0}} \ar@{=}[d] &K_2\oplus X_2 \ar[r]^-{\svech{0}{1}} \ar[d]^{\alpha} &X_2 \ar[d]^{x_2} \ar[r] &0\\
0 \ar[r] &K_2 \ar[r] &B_1 \ar[r] &W_2 \ar[r] &0
}
$$
where $\alpha$ is an epimorphism. Then we have the following commutative diagram
$$\xymatrix{
0 \ar[r] &W_1 \ar[r] \ar@{=}[d] &C_1 \ar[r] \ar[d]^{c_1} &K_2\oplus X_2 \ar[d]^{\alpha} \ar[r] &0\\
0 \ar[r] &W_1 \ar[r] &A_1 \ar[r] &B_1 \ar[r] &0
}
$$
where $c_1$ is an epimorphism. Since we have the following commutative diagram
$$\xymatrix{
0 \ar[r] &W_1 \ar[r] \ar@{=}[d] &W_3 \ar[r] \ar[d] &X_2 \ar[r] \ar[d]^-{\svecv{0}{1}} &0\\
0 \ar[r] &W_1 \ar[r]  &C_1 \ar[r] \ar[d] &K_2\oplus X_2 \ar[d]^-{\svech{1}{0}} \ar[r] &0\\
&&K_2 \ar@{=}[r] \ar[d] &K_2 \ar[d]\\
&&0 &0
}
$$
where $W_3\in \Fac \widehat{\X}$ by Lemma \ref{7-0} and Proposition \ref{733}, we have $C_1\in \T$. By (2), $A_1\in \Fac\T=\T$. Hence $T\in \T*\Fac\mathcal L=\Fac\widehat{\X}*(\Fac\mathcal L*\Fac\mathcal L)=\T$.

By Proposition \ref{CWZ2} and Theorem \ref{PZZ}, $(\mathcal L^-,\mathcal Q^-)$ is a support $\tau$-tilting pair. Since $\Ext^1_{\widehat{\A}}(\widehat{\X},\Fac\widehat{\X})=0$ and $\Ext^1_{\widehat{\A}}(\widehat{\X},\Fac\mathcal L)=0$, we have $\widehat{\X}\in {^{\bot_1}}\T$. Hence $\widehat{\X}\subseteq \mathcal L^-$. Since $\Fac\widehat{\mathcal N}={^\perp(\tau\widehat{\X})}\cap{{\E}^{\perp}}$ by Theorem \ref{last}, we have $\Hom_{\widehat{\A}}(\E,\mathcal L)=0$. Hence $\Hom_{\widehat{\A}}(\E,\T)=0$, which implies that $\Hom_{\widehat{\A}}(\E,\mathcal L^-)=0$. Then $\E\subseteq \mathcal Q^-$. Thus by definition, $(\mathcal L^-,\mathcal Q^-)$ is a support $\tau$-tilting pair that contains $(\widehat{\X},\E)$.
\end{proof}

\begin{rem}
According to \cite{CWZ}, we can call $(\mathcal L^-,\mathcal Q^-)$ the left Bongartz completion of $(\widehat{\X},\E)$ with respect to $\mathcal L$.
\end{rem}

\section{Examples}

In this section we give several examples about our results. Let $k$ be a field.

\begin{exm}
Let $\C$ be a cluster category of type $\mathbb{A}_\infty$ in \emph{\cite{HJ, Ng}}.
It is a $2$--Calabi--Yau triangulated category.
The Auslander--Reiten quiver of $\C$ is as follows:
$$
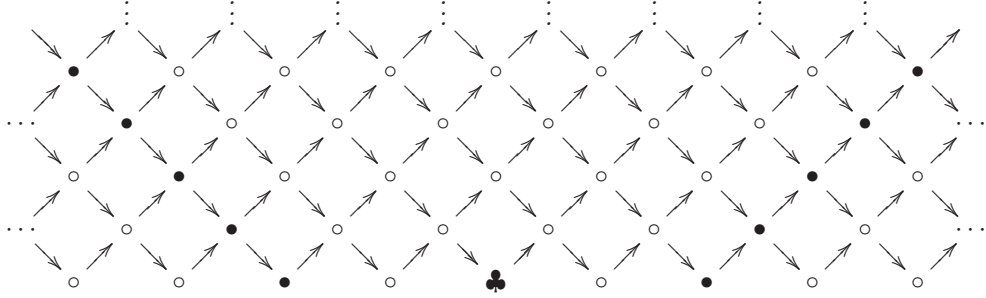
\xymatrix @-2.1pc @! {
    \rule{2ex}{0ex} \ar[dr] & & \vdots \ar[dr]  & & \vdots \ar[dr]& & \vdots \ar[dr]& & \vdots \ar[dr] & & \vdots \ar[dr] & & \vdots \ar[dr] & & \vdots \ar[dr] & & \vdots \ar[dr] && \rule{2ex}{0ex}\\
    & \bullet \ar[ur] \ar[dr] & & \circ \ar[ur] \ar[dr] & &\circ \ar[ur] \ar[dr]  & & \circ \ar[ur] \ar[dr]& & \circ \ar[ur] \ar[dr] & & \circ \ar[ur] \ar[dr] & & \circ \ar[ur] \ar[dr] & & \circ \ar[ur] \ar[dr] && \bullet \ar[ur] \ar[dr] \\
    \cdots \ar[ur]\ar[dr]& & \bullet \ar[ur] \ar[dr] & & \circ \ar[ur] \ar[dr] &  & \circ \ar[ur] \ar[dr] && \circ \ar[ur] \ar[dr] & & \circ \ar[ur] \ar[dr] & & \circ \ar[ur] \ar[dr] & & \circ \ar[ur] \ar[dr] & & \bullet \ar[ur] \ar[dr] &&\cdots\\
    & \circ \ar[ur] \ar[dr] & & \bullet \ar[ur] \ar[dr] & &\circ \ar[ur] \ar[dr] & &\circ \ar[ur] \ar[dr] & & \circ \ar[ur] \ar[dr] & & \circ \ar[ur] \ar[dr] & & \circ\ar[ur] \ar[dr] & & \bullet \ar[ur] \ar[dr] && \circ \ar[ur] \ar[dr] \\
    \cdots \ar[ur]\ar[dr]& & \circ \ar[ur] \ar[dr] &  & \bullet \ar[ur] \ar[dr] && \circ \ar[ur] \ar[dr]& & \circ \ar[ur] \ar[dr] & & \circ \ar[ur] \ar[dr] & & \circ \ar[ur] \ar[dr] & & \bullet \ar[ur] \ar[dr] & & \circ \ar[ur] \ar[dr] &&\cdots\\
    & \circ \ar[ur] & & \circ \ar[ur] & &\bullet \ar[ur] & &\circ \ar[ur] & & \clubsuit \ar[ur] & & \circ \ar[ur] & & \bullet \ar[ur] & & \circ \ar[ur] && \circ \ar[ur]\\ }$$
Let $\X$ the subcategory whose indecomposable objects are marked by $\bullet$ here. Let $\R=\add(\X\cup\{\clubsuit\})$. $\R$ is a maximal rigid subcategory in $\C$. Note that $\R$ is not cluster tilting. $\X$ is two-term $\R[1]$-rigid, and $\R$ is a two-term weak $\R[1]$-cluster tilting subcategory which contains $\X$. But on the other hand, $\R$ is the only two-term weak $\R[1]$-cluster tilting subcategory which contains $\X$. The reason is the following:
\begin{itemize}
  \item[(1)] According to Remark \ref{rem0}, any two-term $\R[1]$-rigid subcategory is rigid.
  \item[(2)] If we want to get a rigid subcategory which contains $\X$ by adding some indecomposable object, $\clubsuit$ is the only choice.
\end{itemize}
This example shows that not every two-term $\R[1]$-rigid subcategory is contained in two different two-term weak $\R[1]$-cluster tilting subcategories. Hence $\R[1]$-functorially finite is a reasonable assumption.

\end{exm}
\vspace{1mm}

\begin{exm}
Let $Q\colon 1\to 2\to 3$ be the quiver of type $A_3$ and $\C:=\mathrm{D}^b(kQ)$ the bounded derived category of $kQ$ whose Auslander-Reiten quiver is the following:
$$
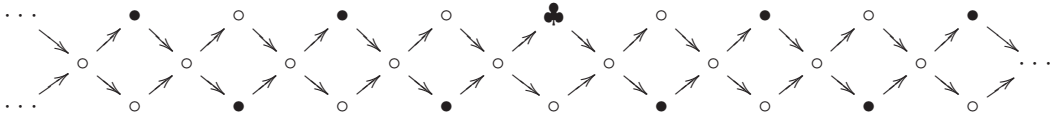
\xymatrix@C=0.3cm@R0.2cm{
\cdot\cdot\cdot \ar[dr] &&\bullet \ar[dr] &&\circ \ar[dr] &&\bullet \ar[dr] &&\circ \ar[dr] &&\clubsuit \ar[dr] &&\circ \ar[dr] &&\bullet \ar[dr] &&\circ \ar[dr] &&\bullet \ar[dr]\\
&\circ \ar[ur] \ar[dr] &&\circ \ar[ur] \ar[dr] &&\circ \ar[ur] \ar[dr] &&\circ \ar[ur]  \ar[dr] &&\circ\ar[ur] \ar[dr]  &&\circ \ar[ur] \ar[dr] &&\circ \ar[ur]  \ar[dr] &&\circ \ar[ur] \ar[dr] &&\circ \ar[ur] \ar[dr] &&\cdot\cdot\cdot\\
\cdot\cdot\cdot  \ar[ur] &&\circ \ar[ur]  &&\bullet \ar[ur]  &&\circ \ar[ur]  &&\bullet \ar[ur]  &&\circ \ar[ur] &&\bullet  \ar[ur]  &&\circ \ar[ur]  &&\bullet \ar[ur]  &&\circ \ar[ur]
}
$$
Let $\X$ be the subcategory whose indecomposable objects are marked by $\bullet$ in the diagram. Let $\R$ be the cluster tilting subcategory whose indecomposable objects are the objects in $\X$ and the object marked by $\clubsuit$. Then $\X$ is an almost complete two-term weak $\R[1]$-cluster tilting subcategory. By Theorem \ref{main1}, $\mathcal N_\X=\R$ and $\M_\X$ is the subcategory whose indecomposable objects are the objects in $\X$ and the object marked by $\spadesuit$ in the diagram:
$$
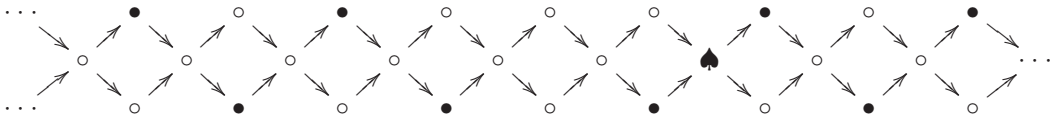
\xymatrix@C=0.3cm@R0.2cm{
\cdot\cdot\cdot \ar[dr] &&\bullet \ar[dr] &&\circ \ar[dr] &&\bullet \ar[dr] &&\circ \ar[dr] &&\circ \ar[dr] &&\circ \ar[dr] &&\bullet \ar[dr] &&\circ \ar[dr] &&\bullet \ar[dr]\\
&\circ \ar[ur] \ar[dr] &&\circ \ar[ur] \ar[dr] &&\circ \ar[ur] \ar[dr] &&\circ \ar[ur]  \ar[dr] &&\circ\ar[ur] \ar[dr]  &&\circ \ar[ur] \ar[dr] &&\spadesuit \ar[ur]  \ar[dr] &&\circ \ar[ur] \ar[dr] &&\circ \ar[ur] \ar[dr] &&\cdot\cdot\cdot\\
\cdot\cdot\cdot  \ar[ur] &&\circ \ar[ur]  &&\bullet \ar[ur]  &&\circ \ar[ur]  &&\bullet \ar[ur]  &&\circ \ar[ur] &&\bullet  \ar[ur]  &&\circ \ar[ur]  &&\bullet \ar[ur]  &&\circ \ar[ur]
}
$$
We can find that $\M_\X$ is indeed two-term weak $\R[1]$-cluster tilting, but it is not cluster tilting (even not rigid).  On the other hand, we can find no more two-term weak $\R[1]$-cluster tilting subcategory which contains $\X$.
\end{exm}

\vspace{1mm}

\begin{exm}
Let $Q$ be the following infinite quiver:
$$\xymatrix@C=0.5cm@R0.5cm{\cdots \ar[r]^{x_{-5}} &-4 \ar[r]^{x_{-4}} &-3 \ar[r]^{x_{-3}} &-2 \ar[r]^{x_{-2}} &-1 \ar[r]^{x_{-1}} &0 \ar[r]^{x_{0}} &1 \ar[r]^{x_1} &2 \ar[r]^{x_2} &3 \ar[r]^{x_3} &4 \ar[r]^{x_4} &\cdots}$$
Let $\Lambda=kQ/(x_ix_{i+1}x_{i+2})$. Then the AR-quiver of $\mod \Lambda$ is the following.
$$
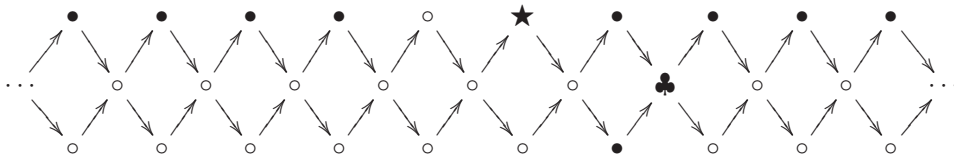
\xymatrix@C=0.2cm@R0.4cm{
&\bullet \ar[dr] &&\bullet \ar[dr] &&\bullet \ar[dr] &&\bullet \ar[dr] &&\circ \ar[dr]  &&\bigstar \ar[dr] &&\bullet \ar[dr] &&\bullet \ar[dr] &&\bullet \ar[dr] &&\bullet \ar[dr]\\
\cdots \ar[ur] \ar[dr] &&\circ \ar[ur] \ar[dr] &&\circ \ar[ur] \ar[dr] &&\circ \ar[ur]  \ar[dr] &&\circ \ar[ur] \ar[dr]  &&\circ \ar[ur] \ar[dr] &&\circ \ar[ur]  \ar[dr] &&\clubsuit \ar[ur] \ar[dr] &&\circ \ar[ur] \ar[dr] &&\circ \ar[ur] \ar[dr] &&\cdots\\
&\circ \ar[ur]  &&\circ \ar[ur]  &&\circ \ar[ur]  &&\circ \ar[ur]  &&\circ \ar[ur] &&\circ  \ar[ur]  &&\bullet \ar[ur]  &&\circ \ar[ur]  &&\circ \ar[ur] &&\circ \ar[ur]
}
$$
Let $\widehat{\X}$ be the subcategory whose indecomposable objects are marked by $\bullet$ in the diagram. Then $(\widehat{\X},0)$ is an almost complete support $\tau$-tilting pair which is contained in exactly two support $\tau$-tilting pairs:
$$(\add(\widehat{\X}\cup\{\clubsuit\}),0), \quad (\add(\widehat{\X}\cup\{\bigstar\}),0)$$
where $\add(\widehat{\X}\cup\{\clubsuit\})=\widehat{\M}$ and $\add(\widehat{\X}\cup\{\bigstar\})=\widehat{\N}$ {\rm (}see Theorem \ref{last}{\rm)}.
\end{exm}

\vspace{4mm}

\hspace{-4mm}\textbf{Data Availability}\hspace{2mm} Data sharing not applicable to this article as no datasets were generated or analysed during
the current study.
\vspace{2mm}

\hspace{-4mm}\textbf{Conflict of Interests}\hspace{2mm} The authors declare that they have no conflicts of interest to this work.


\vspace{4mm}

\begin{thebibliography}{99}

\bibitem[Au]{Au}
M.\ Auslander.
\newblock Representation theory of Artin
algebras I.
\newblock Comm. Algebra 1 (1974), 177--268.

\bibitem[AI]{AI}
T. Aihara O. Iyama.
\newblock Silting mutation in triangulated categories.
\newblock J. Lond. Math. Soc. (2) 85 (2012), no. 3, 633--668.


\bibitem[AIR]{AIR} T. Adachi, O. Iyama, I. Reiten.
$\tau$-tilting theory.
Compos. Math. 150 (2014), no. 3, 415--452.


\bibitem[AST]{AST}
\newblock J. Asadollahi, S. Sadeghi, H. Treffinger.
\newblock On $\tau$-tilting subcategories.
	To appear in Canadian Journal of Mathematics, 2024.
DOI: https://doi.org/10.4153/S0008414X24000221



\bibitem[BIKR]{BIKR} I. Burban, O. Iyama, B. Keller, I. Reiten.
Cluster tilting for one--dimensional hypersurface singularities.
Adv. Math. 217 (2008), no.6, 2443--2484.


\bibitem[BMRRT]{BMRRT} A. B. Buan, R. Marsh, M. Reineke, I. Reiten, G. Todorov.
Tilting theory and cluster combinatorics. Adv. Math. 204 (2006), no.2, 572--618.

\bibitem[BZ]{BZ}
\newblock A. Buan, Y. Zhou.
\newblock Weak cotorsion, $\tau$-tilting and two-term categories. J. Pure Appl. Algebra 228 (2024), no. 1, Paper No. 107445, 18 pp.


\bibitem[CWZ]{CWZ}
P. Cao, Y. Wang, H. Zhang.
\newblock Relative left Bongartz completions and their compatibility with mutations.
\newblock Math. Z. (2023) 305:  27.

\bibitem[CZZ]{CZZ} W. Chang, J. Zhang, B. Zhu.
On support $\tau$-tilting modules over endomorphism algebras of rigid objects.
Acta Math. Sin. (Engl. Ser.) 31 (2015), no. 9, 1508--1516.


\bibitem[FGL]{FGL} C. Fu, S. Geng, P. Liu.
Relative rigid objects in triangulated categories.
J. Algebra 520 (2019), 171--185.


\bibitem[FZ1]{FZ1} S. Fomin, A. Zelevinsky.
Cluster algebras. I. Foundations. J. Amer. Math. Soc. 15 (2002), no.2, 497--529.

\bibitem[FZ2]{FZ2} S. Fomin, A. Zelevinsky.
Cluster algebras. II. Finite type classification. Invent. Math. 154 (2003), no.1, 63--121.

%\bibitem[G]{G} L. Guo.
%Cluster tilting objects in generalized higher cluster categories.
%J. Pure Appl. Algebra 215 (2011), no.9, 2055--2071.


\bibitem[HJ]{HJ}  T. Holm, P. J{\o}rgensen.
\newblock On a cluster category of infinite Dynkin type, and the relation to triangulations of the infinity--gon. \newblock Math. Z. 270 (2012), no. 1--2, 277--295.

\bibitem[HR]{HR} D. Happel, C. M. Ringel.
Tilted algebras.Trans. Amer. Math. Soc. 274 (1982), no.2, 399--443.


\bibitem[HU]{HU} D. Happel, L. Unger.
Almost complete tilting modules. Proc. Amer. Math. Soc. 107 (1989), no.3, 603--610.

\bibitem[IJY]{IJY}
O. Iyama, P. J{\o}rgensen, D. Yang.
\newblock Intermediate co--$t$--structures, two--term silting objects, $\tau$--tilting
modules, and torsion classes.
\newblock Algebra Number Theory 8 (2014), no. 10, 2413--2431.

\bibitem[IY]{IY}
O. Iyama, Y. Yoshino. Mutation in triangulated categories and rigid Cohen--Macaulay modules. Invent. Math. 172 (2008), no. 1, 117--168.



\bibitem[KR]{KR}
B. Keller, I. Reiten.
\newblock Cluster--tilted algebras are Gorenstein and stably Calabi-Yau.
\newblock Adv. Math.211(2007), no.1, 123--151.

\bibitem[KZ]{KZ} S. Koenig, B. Zhu.
From triangulated categories to abelian categories: cluster tilting in a general framework.
Math. Z. 258 (2008), no.1, 143--160.


\bibitem[LC]{LC} J. Le, X. Chen.
\newblock Karoubianness of a triangulated category,
\newblock J. Algebra 310 (2007), no.1, 452--457.

\bibitem[LN]{LN} Y. Liu, H. Nakaoka. Hearts of twin cotorsion pairs on extriangulated categories
J. Algebra 528 (2019), 96--149.


\bibitem[LZ1]{LZ1}
Y. Liu, P. Zhou.
\newblock $\tau$--tilting theory in abelian categories.
\newblock Proc. Amer. Math. Soc. 150 (2022), no. 6, 2405--2413.

\bibitem[LZ2]{LZ2}
Y. Liu, P. Zhou.
\newblock Relative rigid subcategories and $\tau$-tilting theory.
\newblock Algebr. Represent. Theory 25 (2022), no. 6, 1699--1722.


\bibitem[LZ3]{LZ3}
Y. Liu, P. Zhou.
\newblock Cotorsion pairs on $m$-term subcategories.
\newblock To appear in Acta Mathematica Sinica, English Series, 2023.


\bibitem[Na]{Na} H. Nakaoka.
General heart construction on a triangulated category (I): Unifying $t$-structures and cluster tilting subcategories.
Appl. Categ. Structures 19 (2011), no. 6, 879--899.



\bibitem[Ng]{Ng} P. Ng. \newblock A characterization of torsion theories in the cluster category of Dynkin type $\mathbb{A}_\infty$. arXiv: 1005.4364, 2010.

\bibitem[NP]{NP}
H. Nakaoka, Y. Palu.
\newblock Extriangulated categories, Hovey twin cotorsion pairs and model structures.
\newblock Cah. Topol. G\'{e}om. Diff\'{e}r. Cat\'{e}g.  60 (2019), no. 2, 117--193.


\bibitem[PZZ]{PZZ}
\newblock J. Pan, Y. Zhang, B. Zhu.
\newblock Support $\tau$-tilting subcategories in exact categories.
J. Algebra 636 (2023), 455--482.


\bibitem[RS]{RS} C. Riedtmann, A. Schofield.
On a simplicial complex associated with tilting modules. Comment. Math. Helv. 66 (1991), no.1, 70--78.


%\bibitem[T]{T} H. Thomas.
%Defining an $m$-cluster category. J. Algebra 318 (2007), no.1, 37--46.



\bibitem[U]{U} L. Unger.
Schur modules over wild, finite--dimensional path algebras with three simple modules.
J. Pure Appl. Algebra 64 (1990), no.2, 205--222.



%\bibitem[W]{W} A. Wr{\aa}lsen. Rigid objects in higher cluster categories. J. Algebra 321 (2009), no.2, 532--547.

\bibitem[XZ]{XZ} J. Xiao, B. Zhu. Locally finite triangulated categories.
J. Algebra 290 (2005), no. 2, 473--490.

\bibitem[YZ]{YZ}
W. Yang, B. Zhu. \newblock Relaive cluster tilting objects in triangulated categories. Trans. Amer. Math. Soc. 371 (2019), no. 1, 387--412.

\bibitem[YZZ]{YZZ}
W. Yang, P. Zhou and  B. Zhu. \newblock Triangulated categories with cluster--tilting subcategories. Pacific J. Math. 301 (2019), no. 2, 703--740.

\bibitem[ZhZ]{ZhZ} P. Zhou, B. Zhu. Two--term relative cluster tilting subcategories, $\tau$--tilting modules and silting subcategories. J. Pure Appl. Algebra 224 (2020), no. 9, 106365, 22 pp.

\bibitem[ZZ]{ZZ} Y.\ Zhou, B.\ Zhu. \newblock Maximal rigid subcategories in $2$--Calabi--Yau triangulated categories.
\newblock J. Algebra 348 (2011), 49--60.

%\bibitem[ZZ2]{ZZ2} Y.\ Zhou, B.\ Zhu.
%Cluster combinatorics of $d$--cluster categories.
%J. Algebra 321(2009), no.10, 2898--2915.

\end{thebibliography}
\end{document}